\newif\iffullver\fullvertrue

\iffullver
\documentclass[11pt,envcountsame]{llncs}
\usepackage[a4paper,hmargin=3.4cm,vmargin=3cm]{geometry} 
\else
\documentclass[envcountsame]{llncs}
\fi

\usepackage{hyperref}
\usepackage[british]{babel}
\usepackage{enumerate, longtable}
\usepackage{amsmath, amscd, amsfonts, amssymb, latexsym, comment, stmaryrd, graphicx,color,multirow,hhline}
\usepackage[all]{xypic}
\usepackage[T1]{fontenc}
\usepackage[utf8]{inputenc}
\usepackage{algorithm,algcompatible}
\usepackage[ruled,vlined,linesnumbered,algo2e,resetcount]{algorithm2e}
\usepackage[cal=euler,calscaled=1,bb=esstix,bbscaled=1]{mathalfa}
\usepackage{cite}

\usepackage{pgfplotstable}
\pgfplotstableset{
	empty cells with={---},
	every head row/.style={before row=\toprule,after row=\midrule},
	every last row/.style={after row=\bottomrule}
}
\pgfplotsset{compat=1.9}

\usepackage{array,booktabs}
\usepackage{filecontents}

\newcommand{\Hom}{{\rm Hom}}
\newcommand{\GL}{\mathrm{GL}}

\newcommand{\Vol}{{\rm Vol}}

\newcommand{\calH}{\mathcal{H}}
\newcommand{\calR}{\mathcal{R}}
\newcommand{\calL}{\mathcal{L}}

\newcommand{\cM}{\mathcal{M}}
\newcommand{\cN}{\mathcal{N}}
\newcommand{\cO}{\mathcal{O}}

\newcommand{\cS}{\mathcal{S}}

\newcommand{\cU}{\mathcal{U}}

\newcommand{\cX}{\mathcal{X}}

\newcommand{\fp}{\mathfrak{p}}

\newcommand{\fB}{\mathfrak{B}}
\newcommand{\fC}{\mathfrak{C}}

\newcommand{\fX}{\mathfrak{X}}

\newcommand{\FF}{\mathbb{F}}

\newcommand{\NN}{\mathbb{N}}

\newcommand{\QQ}{\mathbb{Q}}
\newcommand{\RR}{\mathbb{R}}

\newcommand{\ZZ}{\mathbb{Z}}

\title{The Hidden Lattice Problem}

\author{}
\institute{}
\author{Luca Notarnicola, Gabor Wiese\thanks{University of Luxembourg, Department of Mathematics, Maison du Nombre 6, Avenue de la Fonte, L-4364 Esch-sur-Alzette, Luxembourg}}
\institute{
	\email{notarnicola.luca@internet.lu},
	\email{gabor.wiese@uni.lu} \\[0.5cm] \today
}

\begin{document}
	\maketitle

\begin{abstract}
We consider the problem of revealing a small hidden lattice from the knowledge of a low-rank sublattice modulo a given sufficiently large integer -- the {\em Hidden Lattice Problem}.
A central motivation of study for this problem is the Hidden Subset Sum Problem, 
whose hardness is essentially determined by that of the hidden lattice problem.
We describe and compare two algorithms for the hidden lattice problem: 
we first adapt the algorithm by Nguyen and Stern for the hidden subset sum problem, based on orthogonal lattices, 
and propose a new variant, which we explain to be related by duality in lattice theory. 
Following heuristic, rigorous and practical analyses,
we find that our new algorithm brings some advantages
as well as a competitive alternative for algorithms for problems with cryptographic interest, 
such as Approximate Common Divisor Problems, 
and the Hidden Subset Sum Problem.
Finally, we study variations of the problem and highlight its relevance to cryptanalysis.
\end{abstract}

\subsubsection*{Keywords.} Euclidean Lattices, Lattice Reduction, Cryptanalysis, Approximate Common Divisor Problem, Hidden Subset Sum Problem

\section{Introduction}

The Hidden Subset Sum Problem asks to reveal a set of binary vectors from a given linear combination modulo a sufficiently large integer. 
At Crypto 1999, 
Nguyen and Stern have proposed an algorithm for this problem, based on lattices, \cite{ns99}. Their solution crucially relies on revealing, in the first place, the ``small'' lattice generated 
by the binary vectors: this is the underlying {\em Hidden Lattice Problem} (HLP).
The starting point of this work is to investigate
the HLP independently.
For this article, we define ``small'' lattices as follows.

\begin{definition}\label{def:small}
Let $0<n\leq m$ be integers.  
Let $\Lambda \subseteq \ZZ^m$ be a lattice of rank~$n$ 
equipped with the Euclidean norm.
We define the {\em size} of a basis $\fB$ of $\Lambda$ by 
$$\sigma(\fB) = \sqrt{ \frac{1}{n}\sum_{v \in \fB}\Vert v\Vert^2} .$$
For $\mu \in \RR_{\geq 1}$, we say that $\Lambda$ is {\em $\mu$-small} if $\Lambda$ 
possesses a basis $\fB$ of size $\sigma(\fB) \le \mu$.
\end{definition}

Small lattices naturally occur in computational problems in number theory and 
cryptography. 
For $\Lambda$ as in Def.~\ref{def:small}, we let
$\Lambda_{\QQ}$ (resp.~$\Lambda_{\RR}$) be the $\QQ$-span (resp.~$\RR$-span) of $\Lambda$ in $\RR^m$.
The completion
$\overline{\Lambda}$ of $\Lambda$ is 
$\Lambda_{\QQ}\cap \ZZ^m = \Lambda_{\RR}\cap \ZZ^m$
and we say that $\Lambda$ is complete if
$\overline{\Lambda}=\Lambda$.
As is customary in many computational problems we also
work modulo $N \in \ZZ$
and
write $v \in \Lambda \pmod{N}$ if there 
exists $w \in \Lambda$ such that $v-w \in (N\ZZ)^m$.
If $\Lambda' \subseteq \ZZ^m$, 
then $\Lambda' \subseteq \Lambda \pmod{N}$ 
shall mean $v \in \Lambda \pmod{N}$ for all $v \in \Lambda'$.
We then define the Hidden Lattice Problem as follows.

\begin{definition}\label{defi:HLP}
Let $\mu \in \RR_{\geq 1}$, integers $1 \leq r \leq n \leq m$ 
and $N \in \ZZ$. 
Let $\calL \subseteq \ZZ^m$ be a 
$\mu$-small lattice of rank~$n$.
Further, let $\cM \subseteq \ZZ^m$ be a lattice
of rank $r$ such that $\cM \subseteq \calL \pmod{N}$. 
The {\em Hidden Lattice Problem (HLP)} is the task 
to compute from the knowledge of $n,N$ and a basis of $\cM$, 
a basis of the completion 
of any $\mu$-small lattice $\Lambda$ of rank $n$ such 
that $\cM \subseteq \Lambda \pmod{N}$.
\end{definition}

Since $\cM$ is defined modulo $N$, 
we may view $\cM \subseteq (\ZZ/N\ZZ)^m$.
We analyse for which values of $\mu \in \RR_{\geq 1}$
a generic HLP 
can be expected to be solvable.
Random choices of $\cM$ are likely to uniquely determine the lattice $\calL$, thus $\Lambda = \calL$. 
We will see that $\overline{\calL}$ 
is very often equal to~$\calL$: 
it is the {\em hidden lattice} to be uncovered (note that the completion makes the lattice only smaller).
Our definition is more general than the framework in \cite{ns99} and deviates in two ways:
first, we do not require $\calL$ to possess a basis of {\em binary} vectors
as in \cite{ns99}, but instead control the size of $\calL$ by $\mu$.
Also, instead of assuming a unique vector to be public ($r=1$), we assume a (basis of a) sublattice $\cM$ of arbitrary rank $r$ to be public.

\subsection{Our contributions} 

Our principle aims are to describe 
algorithms for the HLP, analyse them theoretically,
heuristically and practically, and give applications.

\subsubsection*{Algorithms for the HLP.}
We describe two algorithms for the HLP. 
First, we adapt the orthogonal lattice algorithm 
of Nguyen and Stern \cite{ns99},
based on the (public) lattice 
$\cM^{\perp_N}$ of vectors orthogonal to
$\cM$ modulo $N$. It naturally contains the relatively
small lattice $\calL^{\perp}$, which we can
identify by lattice reduction, 
provided that the parameters satisfy certain conditions.
Our major contribution is to propose a new two-step alternative algorithm,
based on the (public) 
lattice $\cM_N$ of vectors that lie in $\cM$ 
modulo $N$.
In this case, we first explain how to recognize vectors 
lying directly in a relatively small sublattice of the completion of the hidden lattice $\calL$
and compute them by lattice reduction.
We explain that the second step of our new algorithm can be designed to only perform linear algebra over finite fields, which is generally very fast.
Therefore, our second algorithm is often faster than the orthogonal lattice algorithm. 
As we can directly compute short vectors in $\calL$ instead of $\calL^\perp$ 
(which avoids the computation of orthogonal complements), it is also conceptually easier than the orthogonal lattice algorithm.
We finally justify that both algorithms are 
related by duality. 
Using celebrated transference results for the successive minima of dual 
lattices, we explain how to bridge both algorithms
theoretically.
Throughout this paper, we refer to both algorithms as Algorithm I and Algorithm II, respectively.
In cryptanalysis,
the orthogonal lattice has been used extensively, since its introduction in \cite{mhrev}. Lattice duality has been used for example in the context of the LWE Problem, see e.g.~\cite{dual_lwe}.

\subsubsection*{Analysis of our algorithms.}
We provide a heuristic analysis of our algorithms based 
on the Gaussian Heuristic for ``random lattices''.
For Algorithm I, we follow the intuition of \cite{ns99}:
short enough vectors $u \in \cM^{\perp_N}$ (which we compute by lattice reduction) 
must lie in $\calL^\perp$.
Since $\calL^\perp$ has rank $m-n$, we expect to find $m-n$ such vectors. 
For Algorithm II, we derive an explicit lower bound on the 
norm of the vectors lying in $\cM_N$ but outside $\calL_\QQ$, 
which gives us a criterion for establishing an explicit parameter selection. 
In both cases, it turns out that the HLP is solvable 
when $N$ is
is sufficiently large with respect to $\mu$.
Quantifying this difference theoretically and practically is a natural question.
For example, both algorithms detect hidden lattices 
of size $\mu = O(N^{\frac{r(m-n)}{nm}})$ 
up to some terms which differ according to the algorithm; in the balanced case $m=2n=4r$, 
this gives $\mu = O(N^{1/4})$.
To quantify the dependence between $N$ and $\mu$ in a compact formula, 
we propose a definition for an arithmetic invariant attached to the HLP, which we justify to behave like an inverse-density, 
a handy and well-studied 
invariant for knapsack-type problems (see e.g.~\cite{cssplo,ns99}).

We next establish proven results for the case $r=1$, 
not conditioned on the Gaussian Heuristic. 
Such formal statements are not included in \cite{ns99}.
For our proofs, we rely on a discrete counting technique.
For a fixed $\mu$-small basis $\fB$ of $\calL$ (sampled from some set of collections of vectors)
and a given integer $N$, 
we denote by $\calH(\fB)$ a finite 
sample set of vectors
constructed from $\fB$ and $N$.
To an element of $\calH(\fB)$, we naturally associate a HLP
with hidden lattice $\calL$. 
On each of these problems, we ``run'' either Algorithm I or Algorithm II, 
and 
``count'' how often our algorithm
successfully computes 
a basis of $\overline{\calL}$ 
by using LLL, \cite{lll82}.\\

At informal level, we can state the following simplified lower bounds for $\log(N)$ in our heuristic and proven analyses. 
In the proven case (for $r = 1$), the lower bound stands for $\log(N\varepsilon)$, where 
$\varepsilon \in (0,1)$ is fixed such that the success rate of the algorithms is $1-\varepsilon$.
Here $\iota$ denotes the root Hermite factor depending on the chosen lattice reduction algorithm (which is LLL in our proven analysis).

\begin{center}
\begin{table}[h]
	\centering
	\renewcommand{\arraystretch}{2.25}
	\begin{tabular}{l c c cc} \hline
		 & & Algorithm & Lower bound for $\log(N)$& \\[0.5ex]
		\hline\hline 
		&&I & $\log(N) > \frac{mn}{r} \log(\iota)  + \frac{mn}{r(m-n)} \log(\mu) + \frac{n}{2r}\log(m-n)$& \\
		\raisebox{2ex}{\textbf{Heuristic}} & &II & $\log(N) > \frac{m}{m-n} \frac{mn}{r} \log(\iota) + \frac{mn}{r(m-n)}\log(\mu) + \frac{mn}{2r(m-n)}\log\left(\frac{n}{m} \right)$  &\\[2.5ex]
		&&I & $\log(N\varepsilon) > mn\log(\iota) + n(n+1)\log(\mu)  
		+ \frac{n(m-n)}{2}\log(\frac{2(m-n)}{3}) + n\log(3\sqrt{n})
		+1$  & \\
		\raisebox{2.5ex}{\textbf{Proven}} & &II & $\log(N\varepsilon) > mn\log(\iota) + n(n+2)\log(\mu) + n\log(3n^2) + 1$&\\
		\hline 
	\end{tabular}
	\label{tab:Summary_bounds}
	\caption{Lower bounds for $\log(N)$ as functions of $n,m,r,\mu$} 
\end{table}
\end{center}

We have implemented our algorithms in SageMath \cite{sage_92}. 
Our practical results confirm our theoretical findings quite accurately. 
Moreover, we see that both algorithms practically 
perform equivalently well, which is heuristically understandable from the duality between them.
In some cases, Algorithm II outperforms Algorithm I:
the second step of Algorithm II is computationally simpler than for 
Algorithm I, leading to strongly improved running times.  


\subsubsection*{Variations and applications.}
Some variations of Def.~\ref{defi:HLP} are of interest to us.
First, we study the case where 
given vectors lie in a small lattice modulo $N$
only up to unknown short ``noise'' vectors; we 
call this the \textit{noisy hidden lattice problem} (NHLP). 
We notice that we can cancel the effect of the noise, by reducing the NHLP to a HLP with a ``larger'' (in the sense of size and dimension) hidden lattice, and 
apply our previous algorithms without changes.
We also consider a \textit{decisional} version (DHLP) 
of the hidden lattice problem, asking
about the existence of a $\mu$-small lattice $\calL$
containing $\cM$ modulo $N$.
This problem, although not asking for the computation of $\overline{\calL}$ lies at the heart of many 
cryptanalytic settings, 
and may thus be of interest to cryptanalysts.
We recognize that the existence of such $\calL$ strongly 
impacts the geometry of $\cM^{\perp_N}$ (or $\cM_N$)
and, consequently, our algorithms 
solve the decisional version heuristically.

Finally, we describe applications of the HLP together 
with some improvements implied by our Algorithm II. 
Our applications show that the HLP appears somewhat naturally in many different frameworks.
We mostly refer to the works \cite{corper, hssp_cg20, cltindslots, fault_attacks_emv_sign, crypt_rsacrt}. 

\section{Background and notation on lattices}
\label{s:background_lattices}

\subsubsection{Lattices.}
Throughout this section we fix a lattice 
$\Lambda \subseteq \ZZ^m$ of positive rank $n$.
We denote by $\Vol(\Lambda)$ its volume and by
$\lambda_i(\Lambda)$, for $1 \leq i \leq n$,
its successive minima.
Minkowski's Second Theorem \cite[Ch.~2, Thm.~5]{nvlllsa} states that 
\begin{equation}\label{eq:minko_2}
\left(\prod_{i=1}^r \lambda_i(\Lambda)\right)^{1/r}
 \leq \sqrt{\gamma_n}\cdot \Vol(\Lambda)^{1/n} \ , \ 1 \leq r \leq n \ ,
\end{equation}
where $\gamma_n$ is Hermite's constant; one has $\gamma_n \leq \frac{2}{3}n$ for $n \geq 2$, \cite{lagariaskzbases}.
We also have $\Vol(\Lambda) \leq \prod_{i=1}^n \lambda_i(\Lambda)$. 
\begin{definition}\label{defi:ortho}
For $N \in \ZZ$ we call
$\Lambda^{\perp_N} =
 \{ v \in \ZZ^m \;|\; \forall\, w \in \Lambda \pmod{N}:\, \langle v,w \rangle \equiv 0 \pmod{N}\}$
the {\em $N$-orthogonal lattice of $\Lambda$}, and 
$\Lambda_N = \{ v \in \ZZ^m \;|\; v \in \Lambda \pmod{N}\} = \Lambda + N\ZZ^m$
the {\em $N$-congruence lattice of $\Lambda$}.
\end{definition}

The lattices $\Lambda^{\perp_N}$ and $\Lambda_N$ only depend on $\Lambda$ modulo $N$ since $(\Lambda_N)_N=\Lambda_N$ and $(\Lambda_N)^{\perp_N}=\Lambda^{\perp_N}$. 
Therefore, we use the same notation for subgroups $\Lambda \subseteq (\ZZ/N\ZZ)^m$, and mean $\Lambda_N := (\pi_N^{-1}(\Lambda))_N = \pi_N^{-1}(\Lambda)$ and 
$\Lambda^{\perp_N} := (\pi_N^{-1}(\Lambda))^{\perp_N}$, 
where $\pi_N: \ZZ^m \to (\ZZ/N\ZZ)^m$ is the natural projection.
Note $\Lambda^{\perp_0} = \Lambda^{\perp}$, the usual orthogonal lattice,
and $\Lambda_0=\Lambda$.
Assume now $N\neq 0$. 
The map $\Lambda^{\perp_N} \to \Hom_{\ZZ}(\Lambda_N,N\ZZ)
\simeq N\cdot \Hom_{\ZZ}(\Lambda_N,\ZZ) \simeq N \cdot (\Lambda_N)^\vee$
sending $w$ to $(v\mapsto \langle v,w\rangle)$
is an isomorphism.
For a basis matrix $A \in \ZZ^{m\times n}$
of $\Lambda$,
$\Lambda^{\perp_N}$ is the kernel of 
$\ZZ^m\to (\ZZ/N\ZZ)^n, v \mapsto A^T v$
and thus has volume dividing $N^n$.
Since the product of the volumes of dual lattices is $1$, 
we conclude that $N^{m-n}$ divides $\Vol(\Lambda_N)$.
If the Gram matrix $A^TA$ of $A$ is invertible over $\ZZ/N\ZZ$, then we have 
equalities 
$\Vol(\Lambda^{\perp_N})=N^n$ 
and 
$\Vol(\Lambda_N) = N^{m-n}$.

\begin{lemma}\label{lem:lattice}
	Let $\Lambda$ be a lattice of positive rank $n$.
\begin{enumerate}[(a)]
\item The completion of $\Lambda$ 
satisfies $(\Lambda^\perp)^\perp = \overline{\Lambda}$.
\item If $\Lambda' \subseteq \Lambda^\perp$ is a sublattice of the same rank 
as $\Lambda^\perp$, then $\overline{\Lambda} = \Lambda'^\perp$.
\item (Hadamard) $\Vol(\Lambda) \le \prod_{v \in \fB} \Vert v\Vert \le \sigma(\fB)^n$ for any basis~$\fB$ of $\Lambda$.
\item $\Vol(\Lambda^\perp) = \Vol(\overline{\Lambda}) \leq \Vol(\Lambda)$
\end{enumerate}
\end{lemma}

\begin{proof}
For (a) and (d), see Sec.~2 and Cor.~2 in \cite{mhrev}. 
The inequality in (d) follows because $\Lambda \subseteq \overline{\Lambda}$.
Statement (b) follows from (a) because
$\overline{\Lambda}= (\Lambda^\perp)^\perp \subseteq \Lambda'^\perp$ 
are of the same rank with $\overline{\Lambda}$ complete.
The last inequality in (c) follows from the arithmetic-geometric mean inequality. \qed
\end{proof}

\subsubsection{Lattice reduction.}
\label{ss:algos_for_lattices}
Let $\Lambda$ be a lattice of positive rank $n$ in $\ZZ^m$. 
We rely on lattice reduction;
given as input a basis of $\Lambda$, 
a lattice reduction algorithm outputs a reasonably 
short basis of $\Lambda$.
In practice, one often uses LLL \cite{lll82} or BKZ \cite{analyzing_blockwise_lattice_reduction}. 
From a theoretical perspective, BKZ gives slightly better approximation factors than LLL,
but it is widely known that lattice reduction performs 
much better in practice 
than what theory predicts (see e.g.~\cite{gama_pred_latt_red,LLL_average,bkz}).
We summarize the behaviour of LLL below, following \cite{lll82}.

\begin{theorem}\label{theo:LLL}
	Let $\{b_i\}_i$
	be a basis of $\Lambda$. 
	Let $\delta \in (1/4,1)$ and $c=1/(\delta-1/4)$.
	The LLL algorithm 
	with reduction parameter 
	$\delta$ outputs 
	a basis $\{b'_i\}_i$ of $\Lambda$ such that 
	$\Vert b'_j \Vert \leq c^{(n-1)/2} \lambda_i(\Lambda)$
	for all $1 \leq j \leq i \leq n$.
\end{theorem}

Let 
$\{b_i\}_{1\leq i \leq n}$ be a basis of $\Lambda$
with Euclidean norms at most $X \in \ZZ_{\geq 2}$. 
Recall from \cite[Cor.~17.5.4]{gal_math_of_pkc} 
that LLL computes, on input 
$\{b_i\}_{1\leq i \leq n}$,
a reduced basis of $\Lambda$, in
$O(n^5m\log(X)^3)$ bit operations.
In \cite{quadratic_LLL}, using the $L^2$-variant of the LLL algorithm, 
the complexity was improved to 
$O(n^4m(n+\log(X))\log(X))$, quadratic in $\log(X)$ 
(hence the name $L^2$) and based on naive integer multiplication.
See also \cite{quasilinear_LLL} for a variant of LLL with complexity quasi-linear in $\log(X)$.
In this article, we mainly rely on the $L^2$ algorithm with naive integer multiplication 
when analyzing the complexity using LLL.

Whenever we make a heuristic analysis later, we 
assume that a lattice reduction algorithm
outputs a basis $\{b_i'\}_i$ of $\Lambda$ with
$$
\Vert b_i' \Vert \leq \iota^n \lambda_i(\Lambda) \quad , \quad 1 \leq i \leq n\ ,
$$
where the root Hermite factor $\iota > 1$ depends on the reduction algorithm.
By Thm.~\ref{theo:LLL}, $\iota^n=c^{(n-1)/2}$ for LLL,
and $\iota^n=1/2(\gamma_{\beta})^{\frac{n-1}{\beta-1}} (i+3)^{1/2}$ for 
BKZ with block-size $\beta \geq 2$, \cite{schnorr_hierarchy_latt_reduction}.
Heuristically, we can bound the complexity of BKZ from below, by means of an upper bound on $\iota$.
Namely, a root Hermite factor $\iota$ 
is (heuristically) achieved within time at least $2^{\Theta({1/\log(\iota)})}$ 
by using BKZ with block-size $\Theta(1/\log(\iota))$, 
see
\cite{analyzing_blockwise_lattice_reduction}.

\section{Algorithms for the HLP}
\label{s:algos_HLP}
We compare two algorithms for the HLP.
The first one follows the orthogonal lattice algorithm
by Nguyen and Stern, \cite{ns99}.
We then propose a variant based on the (scaled) dual lattice, 
and which has some advantages over the first algorithm.\\

Let us introduce some notation.
For a basis $\fB = \{v_1,\ldots,v_n\}$ of~$\calL$, 
consider the coordinate isomorphism $c_\fB : \calL \to \ZZ^n$, sending 
$\sum_{i=1}^n a_i v_i$ to $(a_1,\ldots,a_n) \in \ZZ^n$. 
Let $\pi_N$ be the natural projection $\ZZ^n \to (\ZZ/N\ZZ)^n$, and denote by 
$c_{\fB,N}: \calL \to (\ZZ/N\ZZ)^n$ the composition $\pi_N c_\fB$.
We now assume that $\calL$ is complete, that is, $\calL = \overline{\calL}$.
Then $N\calL = \calL \cap N\ZZ^m$, and thus
we can extend $c_{\fB,N}$ to $\calL_N \to (\ZZ/N\ZZ)^n$, by setting 
$c_{\fB,N}(\ell + Nt) = c_{\fB,N}(\ell)$ for every $\ell \in \calL$ and $t \in \ZZ^m$.
For $\cM \subseteq \calL \pmod{N}$, that is, $\cM \subseteq \calL_N$, 
define $$\cM_\fB = c_{\fB,N}(\cM) \subseteq (\ZZ/N\ZZ)^n\ , $$
the image of $\cM$ under $c_{\fB,N}$. 
For our algorithms, we consider the lattices 
$(\cM_{\fB})^{\perp_N} \subseteq \ZZ^n$
and 
$(\cM_{\fB})_N = \pi_N^{-1}(\cM_\fB) \subseteq \ZZ^n$ 
of rank $n$, respectively.

\subsection{The orthogonal lattice algorithm for the HLP}
\label{ss:ola}
We adapt the algorithm 
from \cite{ns99}.
Given an instance of the HLP 
with notation as in Def.~\ref{defi:HLP}, 
we have
$\calL^\perp  \subseteq \cM^{\perp_N}$.
In imprecise terms, the smallness of $\calL$ implies the smallness of~$\calL^\perp$.
We argue below that in a sufficiently generic case, $\cM^{\perp_N}$ contains 
a sublattice $\cN_{\rm I}$ of $\calL^\perp$ of the same rank.
By Lem.~\ref{lem:lattice}, $\overline{\calL}=\cN_{\rm I}^\perp$ is a solution to the given HLP.
If
$\cN_{\rm I} \subseteq \calL^\perp$,
then $\cN_{\rm I}^{\perp} = \overline{\calL}$.
The {\em orthogonal lattice algorithm} is as follows; 
we refer to it as Algorithm I.

\begin{algorithm}
	\caption{Solve the HLP using the orthogonal lattice (Algorithm I)}
	\label{alg:hlp_ol}
		\flushleft {\bf Parameters:}
		The HLP parameters $n,m,r,\mu,N$ from Def.~\ref{defi:HLP}\\
		{\bf Input:}
		A valid input for the HLP: 
		a basis of $\cM \subseteq \ZZ^m$ of rank $r$ 
		such that 
		$\cM \subseteq \calL \pmod{N}$ 
		where $\calL$ is a $\mu$-small lattice of rank $n$ in $\ZZ^m$ \\
		{\bf Output:}		
		A basis of the lattice $\overline{\calL}$ (under suitable parameter choice)\\
		(1) Compute a basis matrix $B(\cM,N)$ of $\cM^{\perp_N}$\\
		(2) Run a lattice reduction algorithm on the basis 
		$B(\cM,N)$ to compute a reduced basis $u_1,\dots,u_{\ell}$ 
		of $\cM^{\perp_N}$, where $\ell=m-r$ if $N=0$ and 
		$\ell = m$ otherwise;
		order the vectors $\{u_i\}_i$ by increasing norm\\
		(3) Construct the lattice $\cN_{\rm I} = \bigoplus_{i=1}^{m-n} \ZZ u_i$\\
		(4) Compute and return a basis of
		$\cN_{\rm I}^\perp$ (see Sec.~\ref{ss:practical_aspects})
\end{algorithm}

\subsubsection*{Identifying $\calL^{\perp}$.}
\label{ss:identify_Lperp}
The decisive point in this algorithm is that $\cN_{\rm I}$ is expected to lie in $\calL^\perp$ due to the smallness of the latter for the following heuristic argumentation. A more precise discussion follows below.
Recall that $\calL^\perp$ is a ``small'' 
sublattice of $\cM^{\perp_N}$ 
of rank $m-n$. One hence expects that lattice reduction 
identifies $m-n$ linearly independent ``short'' 
vectors in $\cM^{\perp_N}$. 
Indeed, in practice one sees a significant 
jump in the size of the basis vectors after the first $m-n$ vectors, i.e.~$\cN_{\rm I}$ is the unique ``small'' 
sublattice of $\calL^\perp$ 
of rank $m-n$.
In Sec.~\ref{ss:decisional_HLP} we formulate the decisional hidden lattice problem (DHLP), asking for the existence of $\calL$. 
This size jump is exactly what is detected by the algorithm for the decisional version.
Let us note that heuristically a ``short'' vector orthogonal to~$\cM$ modulo~$N$ is genuinely orthogonal over $\ZZ$.
Consequently, if $m-n' > m-n$ ``short'' vectors 
were found by lattice reduction, 
then $\cM$ would lie in a small lattice of rank 
$n' < n$, which is heuristically not the case.

Prop.~\ref{prop:size_algo1} makes the smallness of $\calL^\perp$ precise by 
giving a lower bound for vectors lying outside~$\calL^\perp$.
For $\fB=\{v_1,\dots,v_n\}$, we also define the group homomorphism 
\begin{equation}\label{eq:homo_phi_B}
\Phi_\fB: \cM^{\perp_N} \to \cM_\fB^{\perp_N} \ , \ 
u \mapsto (\langle u,v_i \rangle)_{i=1,\dots,n} \ .
\end{equation}

Indeed, by the linearity of the scalar product, it is easy to see 
that for vectors $u \in \cM^{\perp_N}$, the vector 
$\Phi_{\fB}(u)$ is in $\cM^{\perp_N}_\fB$.
Note that the kernel of $\Phi_\fB$ is $\calL^\perp$, which is independent of the choice of basis $\fB$.
Following \cite{ns99}, for short enough vectors 
in $\cM^{\perp_N}$, 
their image under $\Phi_{\fB}$ inside 
$\cM_\fB^{\perp_N}$ does not become 
significantly longer (since 
$\{v_i\}_i$ generate a small basis);
if these vectors have Euclidean norm less than 
the first minimum of $\cM_{\fB}^{\perp_N}$, they must be zero in $\cM_\fB^{\perp_N}$,
and hence $u \in \calL^\perp$.

\begin{proposition}\label{prop:size_algo1}
	If $u \in \cM^{\perp_N} \setminus \calL^\perp$, then
	$
	\Vert u \Vert \ge \frac{\lambda_1(\cM_\fB^{\perp_N})}{\sqrt{n}\cdot \mu} 
	$
	for any basis $\fB$ of $\calL$.
\end{proposition}

\begin{proof}
	The kernel of $\Phi_\fB$ is $\calL^\perp$ 
	and is independent of $\fB$.
	So, for $u \in \cM^{\perp_N} \setminus \calL^\perp$, we have
	$ \lambda_1(\cM_\fB^{\perp_N})
	 \le \Vert\Phi_\fB(u)\Vert 
	\le \sqrt{\sum_{i=1}^n \Vert u\Vert^2 
		 \Vert v_i \Vert^2} =\Vert u \Vert \sqrt{n}  \sigma(\fB)$
	by the Cauchy-Schwarz inequality. 
	We conclude using $\sigma(\fB) \leq \mu$, 
	as $\calL$ is $\mu$-small. \qed
\end{proof}

\subsection{An alternative algorithm for the HLP}
\label{ss:dual_algo}

We describe a variant of Algorithm I
based on the (public) lattice $\cM_N$ for $N \neq 0$.
For an instance of the HLP as in Def.~\ref{defi:HLP}, we have
$\cM_N \subseteq \calL_N$.
We argue below that in a sufficiently generic case, 
$\cM_N$ contains a sublattice $\cN_{\rm II}$ of $\overline{\calL}$ of rank $n$.
Under the assumption 
$\cN_{\rm II} \subseteq \calL$,
one obtains a solution to the HLP by $\overline{\cN_{\rm II}} = \overline{\calL}$.
Our algorithm is as follows and we refer to it as Algorithm II:

\begin{algorithm}
	\caption{Solve the HLP using the congruence lattice (Algorithm II)}
	\label{alg:hlp_cong}
	\flushleft {\bf Parameters:}
	The HLP parameters $n,m,r,\mu,N$ from Def.~\ref{defi:HLP}\\
	{\bf Input:}
	A valid input for the HLP: 
	a basis of $\cM \subseteq \ZZ^m$ of rank $r$ 
	such that 
	$\cM \subseteq \calL \pmod{N}$ 
	where $\calL$ is a $\mu$-small lattice of rank $n$ in $\ZZ^m$ \\
	{\bf Output:}		
	A basis of the lattice $\overline{\calL}$ (under suitable parameter choice)\\
	(1)	Compute a basis matrix $B'(\cM,N)$ of $\cM_N$\\
	(2) Run a lattice reduction algorithm on the basis 
		$B'(\cM,N)$ to compute a reduced basis $u_1,\dots,u_m$ 
		of $\cM_N$;
		order the vectors $\{u_i\}_i$ by increasing norm\\
	(3) Construct the lattice
		$\cN_{\rm II} = \bigoplus_{i=1}^{n} \ZZ u_i$ \\
	(4) Compute and return a basis of
		$\overline{\cN_{\rm II}}$ (see Sec.~\ref{ss:practical_aspects})
\end{algorithm}

\subsubsection*{Identifying $\overline{\calL}$ by its smallness.}
\label{ss:identify_L}
The key point of the algorithm is the existence of a somewhat small sublattice $Q \subseteq \cM_N$ of rank $n$. 
Its existence makes lattice reduction applied on $\cM_N$ 
output $n$ short vectors.
Let 
\begin{equation}\label{eq:def_Q}
Q := Q_{\fB,N} := c_\fB^{-1}((\cM_\fB)_N) \cap \calL = {c_{\fB,N}}^{-1}(\cM_{\fB}) \cap \calL\ .
\end{equation} 
Note that $Q \simeq (\cM_\fB)_N$ via the isomorphism $c_\fB$.
The following lemma describes some properties of the lattice $Q$.

\begin{lemma}\label{lem:Q}
	(a) The lattice $Q$ is equal to $\cM_N \cap \calL$.\\
	(b) The index $(\calL : Q)$ is a multiple of $N^{n-r}$, and
	equal to $N^{n-r}$ if $\Vol((\cM_\fB)_N) = N^{n-r}$
\end{lemma}
\begin{proof}
	\textit{(a)} To see that $Q \subseteq \cM_N \cap \calL$, it suffices to show that $Q \subseteq \cM_N$.
	For $q \in Q$, we have by definition, $c_{\fB,N}(q) \in \cM_\fB$, so $c_{\fB,N}(q) = c_{\fB,N}(x)$ for some $x \in \cM$. Therfore 
	$c_{\fB,N}(x-q) = 0$, i.e., $q \in x + N\ZZ^m \subseteq \cM_N$.
	
	To see that $\cM_N \cap \calL \subseteq Q$, let $\ell = x + Nt \in \cM_N \cap \calL$ with 
	$\ell \in \calL, x \in \cM$ and $t \in \ZZ^m$. This gives 
	$c_{\fB,N}(\ell) = c_{\fB,N}(x) \in \cM_\fB$, so 
	$\ell \in c_{\fB,N}^{-1}(\cM_\fB) \cap \calL = Q$. \\
	\textit{(b)} The isomorphism $Q \simeq (\cM_\fB)_N$ implies 
	$(\calL : Q) = (\ZZ^n : (\cM_{\fB})_N) = \Vol((\cM_\fB)_N)$, which is a multiple of $N^{n-r}$.
	If $\Vol((\cM_\fB)_N) = N^{n-r}$, then $(\calL : Q) = N^{n-r}$.
	\qed
\end{proof}

The key point is the following general lemma (Lem.~\ref{lem:bound_algo2}).
When applied to 
$\Lambda' = \calL \subseteq  \calL_N = \Lambda$, 
it gives Prop.~\ref{prop:size_algo2}.

\begin{lemma}\label{lem:bound_algo2}
	Let $\Lambda \subseteq \ZZ^m$ be a
	lattice of rank $m$.
	Let $\Lambda' \subseteq \Lambda$ be a sublattice of rank $1\leq 
	n < m$.
	For every basis $\fB'$ of $\Lambda'$ and
	every
	$u \in \Lambda$ with 
	$u \notin \Lambda'_\QQ$,
	we have
	$$\Vert u \Vert \geq \frac{\Vol(\Lambda)}{\prod_{v \in \fB'} \Vert v\Vert \cdot \prod_{i=n+2}^{m}\lambda_i(\Lambda)} \ .$$
\end{lemma}	
\begin{proof}
	Since $u \in \Lambda$ and
	$u \notin \Lambda'_\QQ$,
	$\Lambda' \oplus \ZZ u$ is a sublattice of $\Lambda$ of rank $n+1$.
	There are linearly independent (ordered) vectors $t_1,\ldots,t_m \in \Lambda$ with 
	$\Vert t_j \Vert = \lambda_j(\Lambda)$ for all $j$.
	Since $\{t_j\}_j$ are linearly independent, 
	we can choose $m-n-1$ vectors $t'_1,\ldots,t'_{m-n-1}$
	among $\{t_j\}_j$ such that 
	$\Omega = \Lambda' \oplus (\ZZ u) \oplus
	(\bigoplus_{j=1}^{m-n-1} \ZZ t'_j)$
	is a sublattice of $\Lambda$ of finite index.
	In particular,
	$\Vol(\Lambda) \leq \Vol(\Omega)$.
	Since $\prod_{j=1}^{m-n-1} \Vert t_j' \Vert \leq \prod_{i=n+2}^m \lambda_i(\Lambda)$,
	we obtain by Hadamard's Inequality that the volume of $\Omega$ is 
	upper bounded by 
	$
	(\prod_{v \in \fB'} \Vert v \Vert) \cdot \Vert u \Vert \cdot 
	\prod_{j=1}^{m-n-1}\Vert t_j' \Vert 
	\leq 
	(\prod_{v \in \fB'} \Vert v \Vert) \cdot \Vert u \Vert \cdot 
	\prod_{i=n+2}^m \lambda_i(\Lambda)$. \qed
\end{proof}

\begin{proposition}\label{prop:size_algo2}
	Let $u \in \cM_N \setminus \calL_{\QQ}$. Then we have 
	\begin{equation*}
		\Vert u \Vert \geq \frac{1}{\mu^n}\cdot \frac{\Vol(\calL_N)}{\prod_{i=n+2}^m \lambda_i(\cM_N)}
	\end{equation*}
\end{proposition}
\begin{proof}
	This is Lem.~\ref{lem:bound_algo2}, Lem.~\ref{lem:lattice} (c) and the inclusion $\cM_N \subseteq \calL_N$. \qed
\end{proof}

As a consequence, short enough vectors in $\cM_N$, which we seek by lattice reduction,
must eventually lie in $\calL_\QQ$,
and as they are integral, 
also in $\overline{\calL}$.

\subsection{Relation between the algorithms} 
\label{ss:relation_algos}

Algorithms I and II are related by the duality relations of $\cM^{\perp_N}$ 
and $\cM_N$ pointed out in Sec.~\ref{s:background_lattices}. 
Therefore, the existence of $n$ short vectors in $\cM_N$ leads to the existence of $m-n$ short vectors in $\cM^{\perp_N}$ and vice versa, 
by relying on Banaszczyk’s Transference Theorem, which we recall first.
\begin{theorem}[\cite{banaszczyk_transference}, Thm.~2.1]
	\label{theo:transference}
	For every lattice $\Lambda \subseteq \RR^m$ of  rank $m$, 
	one has for all $1 \leq j \leq m$,
	the inequality
	$
	1\leq \lambda_j(\Lambda) \lambda_{m-j+1}(\Lambda^\vee) \leq m
	$.
\end{theorem}

\begin{proposition}\label{prop:transference_Nary}
	For every lattice $\cM \subseteq \ZZ^m$, the following hold:
	\begin{enumerate}
		\item [(a)]
		$
		\prod_{j=1}^{m-n} \lambda_j(\cM^{\perp_N}) \leq \gamma_m^{m/2} \frac{\Vol(\cM^{\perp_N})}{N^{n}} \prod_{j=1}^{n}\lambda_j(\cM_N)
		$
		\item [(b)] 
		$
		\prod_{j=1}^n \lambda_j(\cM_N) \leq \gamma_m^{m/2} \frac{\Vol(\cM_N)}{N^{m-n}} \prod_{j=1}^{m-n}\lambda_j(\cM^{\perp_N})
		$
	\end{enumerate} 
\end{proposition}
\begin{proof}
(a) Minkowski's Second Thm.~\eqref{eq:minko_2} gives 
\begin{equation}\label{eq:proof_dual}
\prod_{j=1}^{m-n} \lambda_j(\cM^{\perp_N})
\leq \frac{\gamma_m^{m/2} \cdot \Vol(\cM^{\perp_N})}{\prod_{j=m-n+1}^m \lambda_j(\cM^{\perp_N})}\ 
\end{equation}
and we find a lower bound for 
$\prod_{j=m-n+1}^m \lambda_j(\cM^{\perp_N})$.
Thm.~\ref{theo:transference} 
with $\Lambda=\cM^{\perp_N}$ 
and $\Lambda^\vee = N^{-1} \cM_N$ gives 
$
\lambda_j(\cM^{\perp_N}) \lambda_{m-j+1}(\cM_N) \in [N,mN]
$
for all $1 \leq j \leq m$.
Taking the product over $j=m-n+1,\ldots,m$ yields 
$
\prod_{j=m-n+1}^{m} \lambda_j(\cM^{\perp_N})
\prod_{j=1}^{n} \lambda_{j}(\cM_N) \in
 [N^{n},(mN)^{n}]
$
and we conclude by \eqref{eq:proof_dual}. 
To establish (b), we proceed similarly. \qed
\end{proof}

Therefore, an upper bound on the first $m-n$
successive minima of $\cM^{\perp_N}$ implies 
an upper bound on 
the first $n$ successive minima of $\cM_N$, and vice-versa.

\subsection{Practical discussion on Algorithm I and II}
\label{ss:practical_aspects}

Algorithm I reveals $\overline{\calL}$ by means of orthogonal lattices. 
On the other side, Algorithm II is conceptually easier than Algorithm I, in the sense that it recovers $\overline{\calL}$ much more directly. In fact, as explained in Sec.~\ref{ss:dual_algo}, 
Algorithm II solves a ``hidden sublattice problem'' in the first place, by recovering the lattice 
$\cN_{\rm II} \subseteq \overline{\calL}$.
We now detail the different steps of the algorithms with a practical focus.

\subsubsection*{Bases for $\cM^{\perp_N}$ and $\cM_N$.}
Given $N$ and 
a basis for $\cM$,
bases for $\cM^{\perp_N}$ and $\cM_N$ are easily computed.
To compute a basis for $\cM^{\perp_N}$,
given a basis matrix $M \in \ZZ^{r\times m}$ for $\cM$
with row vectors, 
one may proceed as follows:
write
$M=[M_1|M_2]$ with $M_1 \in \ZZ^{r \times m-r}$
and $M_2 \in \ZZ^{r \times r}$.
Let $M_1'$ and $M_2'$ be the reductions of $M_1$ and $M_2$ modulo $N$.
Without loss of generality, we can assume
$M_2' \in \GL(r,\ZZ/N\ZZ)$. Let $M_2'^{-1}$ be its inverse 
and put $\tilde{M}:=(-M_2'^{-1}M_1')^T$.
Then the block matrix 
\begin{equation}\label{eq:basis_ortho}
B(\cM,N)=
\begin{bmatrix}
{1}_{m-r} & \tilde{M} \\
{0}_{r \times m-r}& N \cdot {1}_{r}
\end{bmatrix}
\end{equation}
is a basis matrix for $\cM^{\perp_N}$, where $1$ and $0$ denote the identity and zero matrix in the indicated dimensions.
This is the matrix $B(\cM,N)$ computed in the first step of Algorithm I (see Alg.~\ref{alg:hlp_ol}).
Indeed, 
$u \in \ZZ^{1\times m}$ lies in $\cM^{\perp_N}$ 
if and only if $u$ is orthogonal modulo $N$ to the rows of $M$, 
i.e. $Mu^T \equiv 0_{r \times 1} \pmod{N}$.
Putting 
$u=(u_1,u_2) \in \ZZ^{1 \times m-r}\times \ZZ^{1\times r}$,
this gives
$M_1u_1^T+M_2u_2^T \equiv 0_{r\times 1} \pmod{N}$, 
or equivalently, 
$M_2^{-1}M_1u_1^T+ u_2^T \equiv 0_{r\times 1}  \pmod{N}$,
which over the integers reads as 
$u_2^T = \tilde{M} u_1^T + N\cdot 1_{r}l^T$ for some 
$l^T \in \ZZ^{r\times 1}$.
Thus $u=(u_1,u_2)=(u_1,l) \cdot B(\cM,N)$
is the image of $(u_1,l)$ under $B(\cM,N)$.

A basis for $\cM_N$ is constructed similarly,
or, one may directly use duality: 
if $B$ is a basis matrix for 
$\Lambda \subseteq \QQ^m$ of full rank $m$, 
then a basis matrix for $\Lambda^\vee$ is 
$B^\vee := (B^T)^{-1}$, where the inverse is taken over $\GL(m,\QQ)$.
Since
$\cM_N=N(\cM^{\perp_N})^\vee$, 
a basis matrix $B'(\cM,N)$ for $\cM_N$ is
thus
$NA^\vee$, 
with $A = B(\cM,N)$.

The first steps of our algorithms rely on running lattice reduction on these bases.
Subsequently the lattices $\cN_{\rm{I}}$ and $\cN_{\rm{II}}$ are constructed as indicated.
The second steps differ more substantially. 
We detail these algorithms below.

\subsubsection*{Orthogonal of $\cN_{\rm I}$.} 
In Algorithm I, once a basis for $\cN_{\rm I}$
is constructed,
one computes 
a basis for $\cN_{\rm I}^\perp$.
This can be done using the LLL
algorithm following 
\cite[Thm.~4 and Alg.~5]{mhrev}; also see \cite[Prop.~4.1]{LLLOrtho}.
Generally, for a lattice $\Lambda \subseteq \ZZ^m$ of rank $n$ with basis matrix 
$B \in \ZZ^{n \times m}$ (with basis vectors in rows), 
the technique relies on LLL-reducing the rows 
of 
$
\begin{bmatrix}
K_B\cdot B^T \ | \ 1_{m}
\end{bmatrix}
\in \ZZ^{(m + n) \times m}
$
for a 
sufficiently large  
constant $K_B \in \NN$ depending on $B$,
and then, projecting the first $m-n$ vectors
of the resulting reduced basis on their last $m$ components.
For the computation of $\cN_{\rm I}^\perp$, 
following \cite[Algorithm 5]{mhrev}, it suffices to 
choose the constant
$K_{U} = \lceil 2^{\ell} \prod_{i=1}^{m-n} \Vert u_i \Vert \rceil$
with $\ell = (m-1)/2+n(n-1)/4$ and
where $U$ is a basis matrix of $\cN_{\rm I}$ with row vectors
$\{u_i\}_i$, 
computed in the first part.

\subsubsection*{Completion of $\cN_{\rm II}$.} 
The completion of $\Lambda \subseteq \ZZ^m$ is the lattice $\overline{\Lambda} = \Lambda_\QQ \cap \ZZ^m$,
which is $\{ v \in \ZZ^m \ | \ dv \in \Lambda, \textrm{for } \textrm{some } d \in \ZZ \setminus\{0\}\}$.
In Algorithm II, once a basis for $\cN_{\rm II}$ is constructed,
we compute a basis for
$\overline{\cN_{\rm II}}$.
One may compute $\overline{\cN_{\rm II}}$ as 
$(\cN_{\rm II}^{\perp})^{\perp}$
by using LLL twice, as in \cite[Thm.~4 and Alg.~5]{mhrev}, and the output is then LLL-reduced.

We describe an alternative method, which in practice works well (see Sec.~\ref{s:practical}).
As predicted by Lemma \ref{lem:Q}, the index of $Q$ in $\calL$ is
$N^{n-r}$ in most of the cases. 
In practical experiments with a solvable hidden lattice problem, 
we observe that $\cN_{\rm II}$ is exactly $\cM_N \cap \calL = Q$, thus
$(\calL:\cN_{\rm II})=N^{n-r}$.
Therefore, more directly, we can 
complete $\cN_{\rm II}$ locally at primes $p$ dividing $N$.
For a prime $p$, define the $p$-{\em completion} of $\Lambda \subseteq \ZZ^m$ by
$$\Lambda^{p^\infty}:= \{v \in \ZZ^m \ | \ p^kv \in \Lambda, \textrm{for } \textrm{some } k \in \NN \}\ . $$
Let $B \in \ZZ^{n \times m}$ be a basis matrix (with rows $\{b_i\}_i$) of 
some lattice $\Lambda \subseteq \ZZ^m$ of rank $n$;
assume $p$ divides the index 
$(\overline{\Lambda}:\Lambda)$.
We compute a basis of $\Lambda^{p^\infty}$ as follows.
Let $\overline{B} \in \FF_p^{n \times m}$ 
be the reduction of $B$ modulo $p$;
let $\overline{\alpha} \in \FF_p^n$ be in $\ker(\overline{B})$,
i.e. 
$\overline{\alpha}\overline{B} \equiv 0 \pmod{p}$.
We represent $\alpha
\in \ZZ^n$ by choosing 
the entries of 
$\overline{\alpha}$ by their unique representatives 
in $\ZZ \cap [-p/2,p/2)$. We may assume that one of the coefficients of $\bar{\alpha}$ equals $1$, say the $i$th coefficient. 
Let $x \in \ZZ^m$ such that 
$\alpha B = p x$.
Let $\Lambda' \subseteq \ZZ^m$ be the lattice generated by $B' \in \ZZ^{n \times m}$ where $B'$ is the matrix obtained from $B$ after
replacing the $i$th row of $B$
by 
$x$; then $\Lambda \subseteq \Lambda'$
and 
$\Lambda_\QQ = \Lambda'_\QQ$. 
By the choice of $x$, the rank of $B'$ over $\FF_\ell$ 
for every prime $\ell \neq p$, 
does not decrease. 
We repeat this for every basis vector in the 
$\FF_p$-kernel of $\overline{B}$ and update $B'$ accordingly. 

In Sec.~\ref{s:practical}, we report that the second step for Algorithm II can, in general, be carried out much more rapidly than the second step of Algorithm I.
This also gives an improved total running time for Algorithm II against Algorithm I.

\section{Heuristic analysis of the algorithms}
\label{s:heuristic}

We provide a heuristic analysis and comparison of Algorithms I and II for $N>0$. 
For $N<0$, it suffices to replace $N$ by $-N$ throughout
the analysis. 
We write $\log$ for the logarithm in base $2$.
Prop.~\ref{prop:size_algo1} and \ref{prop:size_algo2} are the keys in our analysis. 

We rely on the Gaussian Heuristic (GH) for the successive minima for random\footnote{see e.g.~\cite{Ajtai} for a precise setting; here we shall mean ``generic'' lattices, i.e.~lattices with no extra assumptions, such as the existence of particularly small sublattices} lattices.
Accordingly, we heuristically approximate
$\lambda_1(\Lambda)$ by 
$\sqrt{\gamma_n} \cdot \Vol(\Lambda)^{1/n}$.
Additionnally, we heuristically assume all the minima to be roughly equal:
\begin{equation}\label{eq:gauss_heuristic}
\lambda_k(\Lambda) \approx \sqrt{\gamma_n} \cdot \Vol(\Lambda)^{1/n} \ , \ 1 \leq k \leq n \ .
\end{equation}
Since $\calL, \calL^\perp, (\cM_\fB)^{\perp_N}, (\cM_\fB)_N$ (contrary to $\cM^{\perp_N}$ and $\cM_N$) 
do not possess ``small'' sublattices, it is reasonable to follow this heuristic for these lattices.
As $n \to \infty$, we will use the approximation $\gamma_n \approx n/(2\pi e)$.

\subsection{Analysis of Algorithm I}
\label{ss:heuristic_analysis_algo1}

Lattice reduction computes short vectors in $\cM^{\perp_N}$;
let 
$u_1,\dots,u_{m-n}$ be the first $m-n$ vectors in a basis of $\cM^{\perp_N}$
output by a lattice reduction
algorithm.
Since $\cM^{\perp_N}$ contains $\calL^{\perp}$, 
one has
$\Vert u_{m-n} \Vert \leq \iota^m \lambda_{m-n}(\calL^{\perp})$ for some $\iota > 1$ depending on the lattice reduction
algorithm. By Prop.~\ref{prop:size_algo1}, if
\begin{equation}\label{eq:criterion_algo1}
\iota^m \lambda_{m-n}(\calL^{\perp}) < \frac{\lambda_1(\cM_\fB^{\perp_N})}{\sqrt{n} \cdot \mu}
\end{equation}
then $u_{m-n} \in \calL^{\perp}$ and since the vectors $\{u_i\}_i$ are ordered by size, 
we obtain
a sublattice 
$\cN_{\rm I} = \bigoplus_{i=1}^{m-n} \ZZ u_i$ of $\calL^{\perp}$ of the same rank.
The orthogonal complement $\cN_{\rm I}^{\perp}$
is then the completion of~$\calL$, by Lem.~\ref{lem:lattice}.

We rely on the Gaussian Heuristic to estimate $\lambda_{m-n}(\calL^{\perp})$ 
and $\lambda_1(\cM_\fB^{\perp_N})$.
Using $\Vol(\calL^{\perp})\leq \Vol(\calL)$
and Lem.~\ref{lem:lattice}, 
we have 
$
\lambda_{m-n}(\calL^{\perp}) \lesssim 
\sqrt{\gamma_{m-n}}
\cdot 
\mu^{{n}/{(m-n)}}
$.
Assuming that
$\mathrm{Vol}( \cM_\fB^{\perp_N})= N^{r}$ (see Sec.~\ref{s:background_lattices}), as holds in the generic case,
we obtain by \eqref{eq:gauss_heuristic}:
$
\lambda_1( \cM_\fB^{\perp_N})
 \approx \sqrt{\gamma_n} \cdot N^{r/n}
$.
Putting the bounds together and approximating $\gamma_n$ by $n/(2\pi e)$ gives
\begin{equation}\label{eq:heuristic_Nrn_Algo1}
N^{r/n} > \iota^m \cdot(m-n)^{1/2}\cdot
\mu^{\frac{m}{m-n}}  \ .
\end{equation}
There are more ways to read such an inequality: 
since our investigation is on the hidden lattice, 
we could either bound $\mu$, the size of the small basis of the hidden lattice $\calL$, 
as a function of the other parameters,
or else, consider $\mu$ as fixed
and bound the modulus $N$ in terms of the
remaining parameters.
Following this latter approach, by taking logarithms,
Eq.~\eqref{eq:criterion_algo1} 
implies:
\begin{eqnarray}\label{eq:logN_Algo1}
\log(N) 
&>& 
\frac{mn}{r(m-n)} \log(\mu) + \frac{mn}{r} \log(\iota) + \frac{n}{2r}\log(m-n) 
\end{eqnarray}
Eq.~\eqref{eq:logN_Algo1} is a heuristic sufficient condition 
that the chosen lattice reduction algorithm outputs
$m-n$ vectors
$u_1,\ldots,u_{m-n} \in \calL^{\perp}$.

\subsection{Analysis of Algorithm II}
\label{ss:heuristic_analysis_algo2}

We present two alternative analyses: a ``direct analysis'' without relying on Prop.~\ref{prop:size_algo2}, 
and one using Prop.~\ref{prop:size_algo2}.

\paragraph{Direct analysis.}
We run lattice reduction on $\cM_N$;
let  
$u_1,\dots,u_{m}$ be the first $n$ vectors of a reduced basis 
of $\cM_N$.
The existence of the hidden lattice $\calL$ implies the existence of the sublattice $Q = \cM_N\cap \calL$ of $\cM_N$ (defined in Eq.~\eqref{eq:def_Q}), which
impacts the geometry of $\cM_N$ in the following way:
the first $n$ minima of $\cM_N$ are heuristically of the same size as the first
$n$ minima of $Q$, and the remaining $m-n$ minima are much larger.
In particular, the first $n$ minima of $\cM_N$ are expected to be significantly smaller than the quantity predicted by Eq.~\eqref{eq:gauss_heuristic}:
$$
\sqrt{\gamma_m}\cdot \Vol(\cM_N)^{1/m} \approx \sqrt{\gamma_m} \cdot N^{1-r/m} \ ,
$$
which would heuristically be a valid approximation if $\cM_N$ were a ``generic'' lattice (i.e.~without the existence of $Q$).
To measure this gap, we introduce a threshold constant $\theta \geq 1$. 
We heuristically expect to have $u_1, \ldots, u_n \in \calL_\QQ$ 
under the condition
\begin{equation}\label{eq:heuristic_II_direct_condition}
\theta  \cdot \Vert u_n \Vert <  \sqrt{\gamma_m} \cdot N^{1-r/m} \ .
\end{equation}
Since $\cM_N$ contains $Q$, 
we have 
$\Vert u_{n} \Vert \leq \iota^m \lambda_{n}(Q)$ for some $\iota > 1$ 
depending on the lattice reduction
algorithm. 
We assume 
$(\calL : Q) = N^{n-r}$ by Lemma \ref{lem:Q} \textit{(b)}.
Then $\Vol(Q)= N^{n-r}\Vol(\calL)$.
Since $\prod_{i=1}^n \lambda_i(Q) \leq \gamma_n^{n/2} \Vol(Q)$, this gives with $\Vol(\calL) \leq \mu^n$, the approximation
\begin{equation}\label{eq:heuristic_prod_minima_Q}
\prod_{i=1}^n \lambda_i(Q) \lesssim \gamma_n^{n/2} \mu^n N^{n-r} \ .
\end{equation}
With the heuristic assumption that the successive minima of $Q$ are roughly of equal size, this implies the heuristic upper bound 
$
\lambda_i(Q) \lesssim \sqrt{\gamma_n} \mu N^{1-r/n} 
$
for 
$1\leq i \leq n$.
It follows that 
\begin{equation}\label{eq:up_bd_un_heuristic}
\Vert u_n \Vert \lesssim \iota^m \sqrt{\gamma_n} \mu N^{1-r/n} \ .
\end{equation}
Consequently, from Eq.~\eqref{eq:heuristic_II_direct_condition}, 
we expect to have $u_1,\ldots, u_n \in \calL_\QQ$ as soon as
$
\theta \iota^m  \sqrt{\gamma_n}  \mu  N^{1-r/n} <   \sqrt{\gamma_m}  N^{1-r/m} 
$.
Taking logarithms, this gives the condition
\begin{equation}\label{eq:heuristic_logN_Algo2_without_prop}
\log(N) >  \frac{mn}{r(m-n)}\log(\mu) + \frac{m}{m-n} \frac{mn}{r} \log(\iota)  + \frac{mn}{r(m-n)}\log\left(\theta\frac{\sqrt{n}}{\sqrt{m}} \right) \ .
\end{equation}

\paragraph{Analysis using Prop.~\ref{prop:size_algo2}.}
Following Prop.~\ref{prop:size_algo2}, we compute a heuristic upper bound for $\Vert u_n\Vert$ 
and a lower bound for the quotient
$N^{m-n}/(\mu^n \prod_{i=n+2}^m \lambda_i(\cM_N))$, as $\Vol(\calL_N) \geq N^{m-n}$.

Eq.~\eqref{eq:up_bd_un_heuristic} gives a heuristic upper bound for $\Vert u_n\Vert$.
To give a lower bound for the quotient
$N^{m-n}/(\mu^n \prod_{i=n+2}^m \lambda_i(\cM_N))$, we find an upper bound for 
$\prod_{i=n+2}^m \lambda_i(\cM_N)$.
Minkowski's Second Theorem gives
$\prod_{i=n+1}^m \lambda_i(\cM_N) \leq \gamma_m^{m/2} N^{m-r} / \prod_{i=1}^{n} \lambda_i(\cM_N)
$, 
where we have assumed that
$\Vol(\cM_N) = N^{m-r}$ (see Sec.~\ref{s:background_lattices}), 
which is the generic case and heuristically (almost) always true.
The first $n$ minima of $\cM_N$ are heuristically equal to the $n$ minima
of $Q$, as $Q$ is heuristically the only relatively small sublattice of $\cM_N$. 
We can heuristically 
consider the upper bound provided in \eqref{eq:heuristic_prod_minima_Q} as a lower bound, too. 
Indeed, since $\Vol(Q) \leq \prod_{i=1}^n \lambda_i(Q) \leq \gamma_n^{n/2} \Vol(Q)$, 
Minkowski's bound in \eqref{eq:heuristic_prod_minima_Q} is loose by a factor at most $\gamma_n^{n/2}$. 
Note also that assuming equality in \eqref{eq:heuristic_prod_minima_Q} is compatible 
with Eq.~\eqref{eq:gauss_heuristic} for the lattice $Q$.

Therefore
$
\prod_{i=1}^n \lambda_i(\cM_N) \approx \prod_{i=1}^n \lambda_i(Q) \approx 
\gamma_n^{n/2} \mu^nN^{n-r} 
$.
This implies that 
$$
\prod_{i=n+1}^m \lambda_i(\cM_N) \lesssim \frac{\gamma_m^{m/2}N^{m-r}}{\gamma_n^{n/2}\mu^n N^{n-r}} \approx \frac{(2\pi e)^{n/2}}{(2\pi e)^{m/2}}\frac{m^{m/2}}{n^{n/2}} \frac{N^{m-n}}{\mu^n}
=:K(m,n,N,\mu)=:K\ .$$
Since we expect $\lambda_i(\cM_N)$ for $n+1 \leq i \leq m$ to be roughly equal, 
we obtain that
$
\prod_{i=n+2}^{m} \lambda_i(\cM_N) \lesssim K^{(m-n-1)/(m-n)}
$.
Thus, we derive the heuristic lower bound
$$
\frac{N^{m-n}}{ \mu^n \prod_{i=n+2}^m \lambda_i(\cM_N)} \gtrsim 
\frac{N^{m-n}}{ \mu^n K^{\frac{m-n-1}{m-n}}} \ ,
$$
which gives,
$$
\frac{N^{m-n}}{ \mu^n \prod_{i=n+2}^m \lambda_i(\cM_N)} \gtrsim 
\frac{N}{\mu^{\frac{n}{m-n}}} \cdot \left(\frac{n^{n/2}}{m^{m/2}}\right)^{\frac{m-n-1}{m-n}} \cdot \sqrt{2\pi e}^{m-n-1}\ .
$$
Combined with Eq.~\eqref{eq:up_bd_un_heuristic}, Prop.~\ref{prop:size_algo2} says that if
$$
\iota^m \sqrt{\gamma_n} \mu N^{1-r/n} < \frac{N}{\mu^{\frac{n}{m-n}}} \cdot  \left(\frac{n^{n/2}}{m^{m/2}}\right)^{\frac{m-n-1}{m-n}} \cdot \sqrt{2\pi e}^{m-n-1}  \ ,
$$
then $u_{n} \in \calL_{\QQ}$ (and thus $\overline{\calL}$). 
Since $\{u_i\}_i$ are ordered by size, 
$\cN_{\rm II} = \bigoplus_{i=1}^n \ZZ u_i$ is a sublattice of 
$\overline{\calL}$ of rank $n$.
Thus, the completion of $\cN_{\rm II}$
is the completion of~$\calL$.
Simplifying and taking logarithms, gives the approximate condition
\begin{eqnarray}\label{eq:logN_Algo2}
\log(N) &>&  \frac{mn}{r(m-n)}\log(\mu) + \frac{mn}{r}\log(\iota) + \frac{n}{2r}\log(n)\\
&&
+\frac{n}{2r}\log\left(\frac{m^{m}}{n^n}\right) - \frac{n(m-n)}{2r}\log(2\pi e) \notag \ ,
\end{eqnarray}
where we have used the mild approximation $m-n-1 \approx m-n$.
Eq.~\eqref{eq:logN_Algo2} is a heuristic sufficient condition that the chosen lattice reduction algorithm
outputs $n$ vectors in $\calL_\QQ \cap \ZZ^m = \overline{\calL}$.

In Sec.~\ref{ss:param_compare_heuristic}, we will see that, asymptotically (as $n \to \infty$), the heuristic bounds for Algorithms I and II perform very similarly.

\subsection{Parameter comparison of Algorithms I and II}
\label{ss:param_compare_heuristic}

In light of Eq.~\eqref{eq:logN_Algo1} 
and Eq.~\eqref{eq:heuristic_logN_Algo2_without_prop} (resp.~Eq.~\eqref{eq:logN_Algo2}) we deduce that if the term in $\log(\mu)$ is dominant, then
$\log(N) > \frac{mn}{r(m-n)}\log(\mu)$,
and therefore  
heuristically 
both algorithms detect $\mu$-small lattices 
of size approximately
\begin{equation}\label{eq:expected_size_mu}
\mu = O(N^{\frac{r(m-n)}{nm}}) \ ,
\end{equation}
when $r,n,m$ are fixed and $N$ tends to infinity.
Since $r<n$ and $m-n<m$, 
the exponent is strictly less than $1$. 
In the balanced case $m=2n=4r$, 
this gives $\mu = O(N^{1/4})$.
Larger values of $r$ 
make the hidden lattice problem easier (as expected)
as it can be solved with a modulus of $r$ times smaller bitsize.
We now turn to a more detailed comparison of Eq.~\eqref{eq:logN_Algo1} and Eq.~\eqref{eq:heuristic_logN_Algo2_without_prop}.
For fixed $m,n,r,\mu$, 
a sufficiently large value of $N$ satisfies \eqref{eq:logN_Algo1}, resp.~\eqref{eq:heuristic_logN_Algo2_without_prop}.
When $m,n,r$ are considered as constants, then the right-hand sides of 
\eqref{eq:logN_Algo1} and \eqref{eq:heuristic_logN_Algo2_without_prop} differ only by a constant.
To study the value of $N$ asymptotically as $n \to \infty$,
we consider $r$ as constant,
and view $m$ as a function of $n$.
The term $\log(\iota)$ is constant and relatively small;
for example, in practice one achieves a root Hermite factor $\iota$ approximately $1.021$ 
for LLL, 
so $\log(\iota) \approx 0.03$ is
of impact only in large dimensions. 
Table \ref{tb:asymptotic_logN} shows three cases: when $m-n = O(1)$ is 
bounded absolutely (independently of $n$), when $m=O(n)$, and last, and when $m=O(n^\ell)$ for $\ell>1$.

\begin{table}[H]
	\begin{center}
		\renewcommand{\arraystretch}{1.3}
		\begin{tabular}{c||c|c}
			 & \multicolumn{2}{c}{$\log(N)$} \\
			 \hline 
			$m$ & ~Algorithm I~ &  ~Algorithm II~\\[0.2pt]
			\hline
			$n+O(1)$ 	& ~$O(\frac{n^2}{r}\log(\mu))$~  &  ~$O(\frac{n^2}{r}\max(\log(\mu),n))$~ \\
			$O(n)$		 &  ~$O(\frac{n}{r}\max(\log(\mu),n))$~ & ~$O(\frac{n}{r}\max(\log(\mu),n))$~   \\
			$O(n^\ell) , \ell \in \RR_{>1}$ & ~$O(\frac{n}{r}\max(\log(\mu),n^\ell))$~   & 
			~$O(\frac{n}{r}\max(\log(\mu),n^\ell ))$~ \\
		\end{tabular}
	\end{center}
	\caption{Asymptotic lower bounds for $\log(N)$ as functions of $n,r,\mu$}
	\label{tb:asymptotic_logN}
\end{table}

When $m-n=O(1)$, our algorithms heuristically require larger (asymptotically equal) values of $N$.
The last line of Table \ref{tb:asymptotic_logN} 
remains meaningful for $\ell = 1$ and recovers the case $m=O(n)$;
we have separated both for better readability.
Alternatively, we may rewrite \eqref{eq:logN_Algo1} and \eqref{eq:heuristic_logN_Algo2_without_prop}
as 
\begin{equation}\label{eq:inverse_density}
\Delta:=
\log\left(\frac{N^{r/n}}{\mu^{m/(m-n)}}\right)
> \Delta_*(n,m,\iota) \quad , 
\end{equation}
where $\Delta_*(n,m,\iota)$ 
with $* \in \{\mathrm{I},\mathrm{II}\}$ (depending on whether the bound stands for Algorithm I or II)
are the functions depending on $n, m$ and $\iota$,
defined by:
\begin{eqnarray}
	\Delta_\mathrm{I}(n,m,\iota) &=& 
	m\log(\iota) + \frac{1}{2}\log(m-n) \label{eq:Delta_I}		\\
	\Delta_\mathrm{II}(n,m,\iota) &=& \frac{m^2}{m-n}\log(\iota) + \frac{m}{m-n}\log\left(\theta \frac{\sqrt{n}}{\sqrt{m}}\right)
\label{eq:Delta_II}
\end{eqnarray}

We consider $\iota$ as a constant once the lattice reduction algorithm is chosen, 
and treat $m = m(n)$ as a function of $n$,
thus we just write $\Delta_*(n)$ as function of $n$ only.
The number $\Delta$ is regarded as an arithmetic
invariant for the (geometric) hidden lattice problem,
depending on all the parameters of the problem.

\begin{remark}\label{rem:density}
	\normalfont
	In the language of knapsack-type problems, 
	$\Delta^{-1}$ is regarded as a density 
	for the HLP.
	Namely, one commonly attributes a 
	density to knapsack-type problems as a measure of their hardness.
	For the classical ``binary'' subset sum problem \cite{cssplo}, 
	asking to reveal $x_1,\ldots,x_n \in \{0,1\}$ 
	from a sum
	$\alpha=\sum_{i=1}^n \alpha_i x_i$ 
	with given $\alpha_1,\ldots,\alpha_n\in \ZZ$,
	the density is $n/\log(\max_{i}\alpha_i)$.
	When the $\{x_i\}_i$ are not binary, 
	\cite{NSDensAtt} 
	argues that this definition is
	not ``complete'' enough, and introduces 
	the ``pseudo-density'' 
	$(\sum_i x_i^2)\cdot n/\log(\max_i \alpha_i)$,
	taking into account the weights $\{x_i\}_i$. 
	In \cite{multiplessp}, the authors study higher dimensional subset sums 
	where $k \geq 1$ equations are given; 
	thereby the density is generalized as 
	$(1/k)\cdot n/\log(\max_i \alpha_i)$.
	For the hidden subset sum problem \cite{ns99} (see also Sec.~\ref{ss:hssp}), 
	asking to reveal vectors 
	$x_1,\ldots,x_n \in \{0,1\}^m$ and 
	weights $\alpha_1,\ldots,\alpha_n$ from a given vector 
	$v \equiv \sum_{i=1}^n \alpha_i x_i \pmod{N}$,
	the density has been defined as $n/\log(N)$,
	which, however, is independent of the dimension $m$.
	In light of this discussion, we believe that the definition 
	of $\Delta^{-1}$ is a more complete definition for a density of the 
	HLP.
	For large enough $m$ (say $m \to \infty$) and $r=1$, 
	our definition \eqref{eq:inverse_density} 
	roughly recovers that of \cite{ns99}
	since 
	$
	\Delta^{-1} \to 1/\log(N^{1/n}/\mu) = 
	n/\log(N/\mu^n)
	$.
	Our bounds show that heuristically our algorithms are more likely to succeed for larger values of $\Delta$ (i.e.~larger gaps between $N$ and $\mu$).
\end{remark}

\begin{proposition}\label{prop:asymptotic_params}
	(a) 
	Let $m = \ell n$ for $\ell > 1$. 
	Then 
	$\Delta_{\mathrm{I}}(n) = O(n)$ and
	$\Delta_{\mathrm{II}}(n) = O(n)$.\\
	(b)
	Let $m=n^\ell$ for $\ell > 1$. 
	Then
	$\Delta_{\mathrm{I}}(n) = O(n^\ell)$ 
	and
	$\Delta_{\mathrm{II}}(n) = O(n^{\ell})$.
\end{proposition}
The proof is immediate from 
growth comparisons in 
\eqref{eq:Delta_I}
and \eqref{eq:Delta_II}.


\subsection{Complexity of lattice reduction}
\label{ss:complexity}
The computations of $\cN_{\rm I}$ and $\cN_{\rm II}$ are carried out by 
lattice reduction. 
We describe their complexity by
the LLL and BKZ algorithm.
We see that the LLL reduction ($L^2$-reduction) step in Algorithm II is faster
than in Algorithm I when $r \geq m/2$.\\

Applying the $L^2$-algorithm in Algorithm I, on a basis of $\cM^{\perp_N}$
given by the matrix $B$ in Eq.~\eqref{eq:basis_ortho}.
The top right block $\tilde{M}$ in $B$ has size 
$(m-r) \times r$
and entries of size at most $N$, 
so, every row in $B$ has Euclidean norm at most
$\max((rN^2+1)^{1/2},N)=(rN^2+1)^{1/2}$.
This gives (see Sec.~\ref{ss:algos_for_lattices}), a complexity 
$O(m^6\log((rN^2+1)^{1/2})+m^5\log^2((rN^2+1)^{1/2}))$
which approximately is 
\begin{equation*}\label{eq:complexity_N1}
O(m^6\log(r^{1/2}N)+m^5\log^2(r^{1/2}N)) \ ,
\end{equation*}
for computing $\cN_{\rm I}$ by the $L^2$-algorithm.
For Algorithm II, the $L^2$-algorithm is run on the basis matrix 
$NB^\vee$ of $\cM_N$, with $N\cdot 1_{m-r}$ in the
top left corner. 
The rows have Euclidean norm at most
$\max(N,((m-r)N^2+1)^{1/2})=
((m-r)N^2+1)^{1/2}$.
This gives 
an approximate time complexity 
\begin{equation*}\label{eq:complexity_N2}
O(m^6\log((m-r)^{1/2}N)+m^5\log^2((m-r)^{1/2}N)) \ ,
\end{equation*}
for computing $\cN_{\rm II}$ by the $L^2$-algorithm.
In particular, this complexity is lower than that for computing $\cN_{\rm I}$ when 
$r\geq m/2$. 
In the case $r=1$, computing $\cN_{\rm I}$ is thus faster than $\cN_{\rm II}$, which we confirm practically in Sec.~\ref{s:practical}.

When the prime factorization of $N$ is known and 
$p$ denotes the smallest prime factor of $N$,
then the complexity 
can be reduced 
by replacing $N$ by $p$ in the 
aforementioned formulae,
provided that 
$\log(p)$ satisfies the (heuristic) bounds \eqref{eq:logN_Algo1} and 
\eqref{eq:heuristic_logN_Algo2_without_prop}, respectively,
and by performing the first steps of the algorithms 
over $\ZZ/p\ZZ$ instead of $\ZZ/N\ZZ$.
Namely, 
$\cM \subseteq \calL \pmod{N}$
implies 
$\cM \subseteq \calL \pmod{p}$,
which in the first step, leads to consider 
the lattices
$\cM^{\perp_p}$ and $\cM_p$, respectively.

When using BKZ lattice reduction, we rely on our heuristic analyses to obtain a 
lower bound on the complexity for computing 
$\cN_{\rm I}$ and $\cN_{\rm II}$.
A root Hermite factor $\iota$ 
is achieved within time at least $2^{\Theta({1/\log(\iota)})}$ 
by using BKZ with block-size $\Theta(1/\log(\iota))$.
For both algorithms,
$\log(\iota) < \frac{r}{mn}\log(N)$
gives a heuristic time complexity
$2^{\Theta\left(mn/(r\log(N))\right)}$ to compute $\cN_{\rm I}$, resp.~$\cN_{\rm II}$,
with BKZ. 

\section{Theoretical analysis by counting}
\label{s:counting}

\subsection{Notation and main results}
We restrict to the most basic case $r=1$. We fix $n,m\in \ZZ_{\geq 2}$ with $m>n$ and $\mu\in \RR_{\geq 1}$,
$N \in \ZZ_{>0}$. 
Let $\Omega:=\Omega(n,m,\mu)$ be the set of collections 
$\fB =\{v_i\}_i$ of $n$ $\ZZ/N\ZZ$-linearly independent vectors in $\ZZ^m$ 
satisfying
$\sigma(\fB):=(\frac{1}{n}\sum_i \Vert v_i\Vert^2)^{1/2} \leq \mu$. 
For $\fB \in \Omega$, let $\calL(\fB)$ be the $\mu$-small 
lattice generated by $\fB$;
this is the ``hidden lattice''. 
Consider the homomorphism $F_{\fB} : (\ZZ/N\ZZ)^n \to (\ZZ/N\ZZ)^m$
sending $a=(a_i)_i$ to $F_\fB(a) = \sum_i a_i \pi_N(v_i)$,
where $\pi_N:\ZZ^m \to (\ZZ/N\ZZ)^m$ is reduction modulo $N$.
Let $\cM(a)$ be the lattice $\ZZ F_\fB(a)$ generated by $F_\fB(a)$.
By construction, $\cM(a) \subseteq \calL(\fB) \pmod{N}$ defines a 
hidden lattice problem, asking to compute a basis of $\overline{\calL(\fB)}$ 
on input $\cM(a)$ and $N$ (and $n$).
We identify $F_{\fB}(a)$ with this problem
and our sample space
for the hidden lattice problems is 
$\calH(\fB) = \{F_\fB(a) \ | \ a \in (\ZZ/N\ZZ)^n\}$.
Clearly, $\#\calH(\fB) = N^n$.
For $\delta \in (1/4,1]$, denote by 
$\calH_{\delta,\mathrm{I}}(\fB) \subseteq \calH(\fB)$ 
(resp.~$\calH_{\delta,\mathrm{II}}(\fB) \subseteq \calH(\fB)$), 
the subset of 
$\calH(\fB)$
for which Algorithm I (resp.~Algorithm II) succeeds by using $\delta$-LLL in the first step.

\begin{theorem}\label{theo:main_I}
	Let $\mu \in \RR_{\geq 1}$ and $m > n \geq 3$ and $N>0$ be integers.
	Let $\delta \in (1/4,1)$, $c=(\delta-1/4)^{-1}$ and
	 $\varepsilon \in (0,1)$ such that
	\begin{eqnarray}\label{eq:condition_main}
	\log(N\varepsilon) &>& \frac{mn}{2}\log(c) + n(n+1)\log(\mu)  \\
	&&+ \frac{n(m-n)}{2}\log((2/3)(m-n)) + n\log(3\sqrt{n})
	+1  \notag
	\end{eqnarray}
	For every $\fB \in \Omega$, at least $(1-\varepsilon)\#\calH(\fB)$
	of the hidden lattice problems from $\calH(\fB)$ are solvable 
	by Algorithm $\mathrm{I}$ with $\delta$-$\mathrm{LLL}$; 
	i.e. 
	$
	\#\calH_{\delta,\mathrm{I}}(\fB)/\#\calH(\fB) \geq 1-\varepsilon 
	$.
\end{theorem}

\begin{theorem}\label{theo:main_II}
	Let $\mu \in \RR_{\geq 1}$ and $m > n \geq 3$ and $N>0$ be integers.
	Let $\delta \in (1/4,1)$, $c=(\delta-1/4)^{-1}$ and
	$\varepsilon \in (0,1)$ such that
	\begin{eqnarray}\label{eq:condition_main_II}
	\log(N\varepsilon) &>& \frac{mn}{2}\log(c) + n(n+2)\log(\mu) 
	 + n\log(3n^2) 
	 + 1 \notag
	\end{eqnarray}
	For every $\fB \in \Omega$, at least $(1-\varepsilon)\#\calH(\fB)$
	of the hidden lattice problems from $\calH(\fB)$ are solvable 
	by Algorithm $\mathrm{II}$ with $\delta$-$\mathrm{LLL}$; 
	i.e. 
	$
	\#\calH_{\delta,\mathrm{II}}(\fB)/\#\calH(\fB) \geq 1-\varepsilon
	$.
\end{theorem}

\begin{corollary}\label{cor:N}
	Let $m > n \geq 3$. For every $\delta \in (1/4,1)$ and $\varepsilon \in (0,1)$, there exist 
	positive real numbers 
	$N^{\dag}_{\mathrm{I}}=N_{\delta,\mu,n,m}(\varepsilon)$ 
	and
	$N^{\dag}_{\mathrm{II}}=N_{\delta,\mu,n,m}(\varepsilon)$
	depending on
	$n,m,\mu,\delta,\varepsilon$,
	such that for all integers 
	$N > \min(N^{\dag}_{\mathrm{I}},N^{\dag}_{\mathrm{II}})$
	and all $\fB \in \Omega$, 
	at least $(1-\varepsilon)\#\calH(\fB)$
	of the hidden lattice problems from $\calH(\fB)$ are solvable
	(by Algorithm $\mathrm{I}$ if 
	$\min(N^{\dag}_{\mathrm{I}},N^{\dag}_{\mathrm{II}}) = N^{\dag}_{\mathrm{I}}$
	and Algorithm $\mathrm{II}$ otherwise)
	using $\delta$-$\mathrm{LLL}$.
\end{corollary}

\subsection{Proof of Theorem \ref{theo:main_I}}
Fix integers\footnote{the condition $n \geq 3$ is used for Lem.~\ref{lem:lower_bound_lambda1}.} $m>n\geq 3$, 
$N>0$ and $\mu \in \RR_{\geq 1}$. 
It is enough to show that under 
the assumption in \eqref{eq:condition_main}, we can compute 
a sublattice $\cN_{\rm I}$ of $\calL(\fB)^\perp$ of rank $m-n$.
A basis for $\cN^\perp_{\rm I}$ then gives a basis of $\overline{\calL(\fB)}$.
To prove Thm.~\ref{theo:main_I} we proceed in three steps. 
Given $a \in (\ZZ/N\ZZ)^n$, we establish a lower bound for $\lambda_1((\ZZ a)^{\perp_N})$
and then an upper bound for $\Vert u_{m-n} \Vert$ where $\{u_i\}_i$ 
is a $\delta$-LLL reduced basis of $(\ZZ F_\fB(a))^{\perp_N}$.
We conclude the proof by combining with 
Prop.~\ref{prop:size_algo1}. 

\subsubsection*{Step 1.}
\label{sss:Step1}
For a lower bound for 
$\lambda_1((\ZZ a)^{\perp_N})$, 
we use a counting argument. 
For $t=(t_1,\ldots,t_n) \in \ZZ^n$, let
$\gcd(t,N):=\gcd(t_1,\ldots,t_n,N)$.

\begin{lemma}\label{lem:count_ortho_modN}
	For every non-zero vector $t \in \ZZ^n$
	with $d=\gcd(t,N)$, one has
	$$
	\#\{a \in (\ZZ/N\ZZ)^n \ | \ \langle a,t\rangle \equiv 0 \pmod{N}\} = dN^{n-1}
	\ .$$
\end{lemma}
\begin{proof}
	If $d=1$, then the set in the statement is the kernel of the surjective 
	(as $\gcd(t,N)=1$) 
	homomorphism $\varphi_t:(\ZZ/N\ZZ)^n \to \ZZ/N\ZZ, a \mapsto 
	\langle a,t\rangle$ with 
	$\#\ker(\varphi_t) = N^{n-1}$.
	If $d> 1$, let $t'=(1/d)t$. Then 
	$\langle a,t \rangle \equiv 0 \pmod{N}$ 
	if and only if $\langle a,t' \rangle \equiv 0 \pmod{N/d}$,
	and we represent $a$ as $a_1 + (N/d) a_2$ with 
	$a_1 \in (\ZZ/(N/d)\ZZ)^n$ and $a_2 \in (\ZZ/d\ZZ)^n$.
	The number of such $a$ with $\langle a_1,t' \rangle \equiv 0 \pmod{N/d}$ is $(N/d)^{n-1} \cdot d^n$. \qed
\end{proof}

For $R>0$, let $B_n(R)$ be the $n$-dimensional closed ball of radius $R$ centered at the origin. Let 
$S_n(R)=\#\{x \in \ZZ^n \ | \ \Vert x \Vert \leq R\}$ the number of integral 
points in $B_n(R)$.
We use the simple upper bound $S_n(R) \leq (2R+1)^n \leq (3R)^n$ if 
$R \geq 1$.

\begin{lemma}\label{lem:lower_bound_lambda1}
	For $\varepsilon \in (0,1)$, let 
	$k_{\varepsilon}:= k_{\varepsilon}(n,N) =
	 \frac{1}{3}(\frac{6\varepsilon}{\pi^2})^{1/n}N^{1/n}
	 $. Then
	\begin{equation*}
	\frac{1}{N^n}\cdot {\#\{a \in (\ZZ/N\ZZ)^n \ | \ \lambda_1((\ZZ a)^{\perp_N}) >
		 k_{\varepsilon}\}}\geq 1-\varepsilon\ .
	\end{equation*}
\end{lemma}
\begin{proof}
	For $R>0$, 
	let 
	$\alpha_n(R) = N^{-n}\cdot{\#\{a \in (\ZZ/N\ZZ)^n \ | \ \lambda_1((\ZZ a)^{\perp_N}) \leq R\}}$; 
	we prove $\alpha_n(k_{\varepsilon}) \leq \varepsilon$.
	Without loss of generality, we let $1 \leq R < N$. As the vectors $\{Ne_i\}_i$ (for the canonical basis $\{e_i\}_i$) of norm $N$ 
	lie in $(\ZZ a)^{\perp_N}$, and so $\alpha_n(R) = 1$ for $R \geq N$.
	Then $N^n \alpha_n(R)
	=\#\{a \in (\ZZ/N\ZZ)^n \ | \ \exists\ t \in 
	B_n(0,R)\cap\ZZ^n \setminus \{0\} , 
	\langle a,t\rangle \equiv 0 \pmod{N}\}$
	is upper bounded by 
	$\sum_{t} \#\{a \ | \ \langle a,t\rangle \equiv 0 \pmod{N}\}$, which is
	\begin{eqnarray}\label{eq:proof_lem_lower_bound_lambda1}
	\sum_{d \mid N, d \neq N} \left(\sum_{t, \gcd(t,N)=d} 
	\#\{a \ | \ \langle a,t\rangle \equiv 0 \pmod{N}\} \right)  \ ,
	\end{eqnarray}
	where $t$ runs over $B_n(0,R)\cap\ZZ^n\setminus\{0\}$. Note that 
	in the outer sum we omit $d=N$ as $\Vert t \Vert \leq R < N$ and therefore
	every entry of $t$ is less than $N$.
	We estimate the number of terms in the inner sum
	for a given divisor $d$ of $N$.
	By dividing every entry of $t$ by $d$
	we have 
	$\#\{t \in \ZZ^n \setminus \{0\}: \Vert t \Vert \leq R , \gcd(t,N)=d\} \leq S_n(0,R/d) \leq 3^n(R/d)^n$, if $R \geq d$. Otherwise, the same bound still holds, because we 
	we count non-zero points.
	Using Lem.~\ref{lem:count_ortho_modN},
	one has 
	$\#\{a \ | \ \langle a,t\rangle \equiv 0 \pmod{N}\} = dN^{n-1}$ for 
	vectors $t$ with $\gcd(t,N)=d$.
	Finally,
	 \eqref{eq:proof_lem_lower_bound_lambda1} is at most
	$$3^n \sum_{d} (R/d)^n(dN^{n-1}) = 3^nR^nN^{n-1} \sum_{d} d^{1-n} 
	\leq 
	3^nR^nN^{n-1}  \sum_{d\geq 1} d^{-2} \ ,
	$$ 
	where for the last inequality we have used $n \geq 3$. 
	This sum equals $3^nR^nN^{n-1} \pi^2/6$ 
	and hence $\alpha_n(R) \leq 3^nR^n \pi^2/(6N)$.
	Taking $R$ equal to $R_\varepsilon := \frac{1}{3}(6N\varepsilon/\pi^2)^{1/n}$
	gives $\alpha_n(R_\varepsilon) \leq \varepsilon$.
	In conclusion, letting $k_\varepsilon = \min(N,R_{\varepsilon}) = R_\varepsilon$, 
	gives the result. \qed
\end{proof}

\subsubsection*{Step 2.}
\label{sss:Step2}
To $(\fB,a) \in \Omega \times  (\ZZ/N\ZZ)^n$, we associate the vector $F_\fB(a)$, 
which we identify with an HLP. 
The first step of Algorithm I 
computes a reduced basis of 
$(\cM(a))^{\perp_N}$. 
For $\delta \in (1/4,1]$, we consider a $\delta$-LLL
reduced basis $\{u_i^{(\fB,a,\delta)}\}_i$ of $\cM(a)^{\perp_N}$.
We establish an upper bound for $\Vert u_{m-n}^{(\fB,a,\delta)} \Vert$, by  
Minkowski's Second Theorem and a counting argument similar to Step 1.
Using $\calL(\fB)^\perp \subseteq (\cM(a))^{\perp_N}$,
$\delta$-LLL (Thm.~\ref{theo:LLL}) outputs vectors $\{u_i^{(\fB,a,\delta)}\}_i$
such that 
\begin{equation}\label{eq:step2_LLL}
\Vert u_{m-n}^{(\fB,a,\delta)} \Vert \leq c^{(m-1)/2}  \lambda_{m-n}(\calL(\fB)^\perp)
\end{equation}
where $c = (\delta-1/4)^{-1}$.
We obtain an upper bound for 
$\lambda_{m-n}(\calL(\fB)^{\perp})$
by Minkowski's Second Theorem (Eq.~\eqref{eq:minko_2}):
\begin{equation}\label{eq:pfI_last_min_Lperp}
\lambda_{m-n}(\calL(\fB)^\perp) \leq \prod_{i=1}^{m-n} \lambda_i(\calL(\fB)^\perp) 
\leq  ((2/3)(m-n))^{(m-n)/2}\Vol(\calL(\fB)^{\perp}) \ ,
\end{equation}
which gives
$\lambda_{m-n}(\calL(\fB)^\perp) \leq ((2/3)(m-n))^{(m-n)/2} \mu^n$,
since we have
$\Vol(\calL(\fB)^{\perp}) \leq \Vol(\calL(\fB))\leq \mu^n$ (Lem.~\ref{lem:lattice}).
This gives for every $a \in (\ZZ/N\ZZ)^n$:
\begin{equation}\label{eq:up_bd_um-n}
\Vert u_{m-n}^{(\fB,a,\delta)}\Vert 
\leq c^{(m-1)/2}  ((2/3)(m-n))^{(m-n)/2} \mu^n \ .
\end{equation}

\subsubsection*{Step 3: Proof of Theorem \ref{theo:main_I}.}
\label{sss:proof_main_theorem}

Let $\fB\in \Omega$ and $\varepsilon \in (0,1)$. 
We continue to use 
the notation $k_\varepsilon$ introduced above.
Eq.~\eqref{eq:condition_main}
implies that $\log(N\varepsilon)$ is strictly larger than 
$\frac{n(m-1)}{2}\log(c) + n(n+1)\log(\mu) 
+ \frac{n(m-n)}{2}\log((2/3)(m-n)) + n\log(3\sqrt{n})+\log(\pi^2/6)
$;
and it is a direct computation 
to see that this is equivalent to  
\begin{equation}\label{eq:main_proof}
c^{(m-1)/2}  ((2/3)(m-n))^{(m-n)/2} \mu^n < k_\varepsilon/(\sqrt{n}\mu)\  .
\end{equation}
By Lem.~\ref{lem:lower_bound_lambda1}, $k_\varepsilon<\lambda_1((\ZZ a)^{\perp_N})$
for at least $(1-\varepsilon)N^n$ choices of $a \in (\ZZ/N\ZZ)^n$.
By Eq.~\eqref{eq:up_bd_um-n}, $c^{(m-1)/2}  ((2/3)(m-n))^{(m-n)/2} \mu^n$ is 
an upper bound
for 
$\Vert u_{m-n}^{(\fB,a,\delta)} \Vert$ where $\{u_i^{(\fB,a,\delta)}\}_i$ is a $\delta$-LLL reduced basis of 
$\cM(a)^{\perp_N}$ for every $a$.
Hence, for at least $(1-\varepsilon)N^n$ 
choices of $a \in (\ZZ/N\ZZ)^n$, \eqref{eq:main_proof}
implies 
$\Vert u_{m-n}^{(\fB,a,\delta)} \Vert  < \lambda_1((\ZZ a)^{\perp_N})/(\sqrt{n} \mu)$.
Prop.~\ref{prop:size_algo1} gives
$u_{i}^{(\fB,a,\delta)} \in \calL(\fB)^{\perp}$ 
for all $1 \leq i \leq m-n$. 
This terminates the proof.

\subsection{Proof of Theorem \ref{theo:main_II}}

Fix integers $m>n\geq 3$, $N>0$ and $\mu \in \RR_{\geq 1}$. 
It is enough to show that under the assumption in \eqref{eq:condition_main_II}, 
we can compute a sublattice $\cN_{\rm II}$ of $\overline{\calL(\fB)}$ of rank $n$.
A basis for $\overline{\cN_{\rm II}}$ then gives a basis of $\overline{\calL(\fB)}$.
To prove Thm.~\ref{theo:main_II}, we again proceed in three steps, similarly to the proof of Thm.~\ref{theo:main_I}.
Given $a \in (\ZZ/N\ZZ)^n$, we first establish an upper bound for 
$\Vert u_n \Vert$, where $\{u_i\}_i$ is a $\delta$-LLL reduced basis of 
$\cM(a)_N$. We conclude the proof using Prop.~\ref{prop:size_algo2}.

\subsubsection*{Step 1.}
\label{sss:Step1_II}

For a given $a \in (\ZZ/N\ZZ)^n$, we consider a $\delta$-LLL reduced basis $\{u_i^{(\fB,a,\delta)}\}_i$ of $\cM(a)$.
Note that by construction $(\cM(a)_\fB)_N = (\ZZ a)_N$.
The lattice $Q(a) = \cM(a)_N\cap \calL(\fB)$ is defined as in Sec.~\ref{s:algos_HLP} (see Lem.~\ref{lem:Q}).
The following lemma gives an upper bound for $\Vert u_n^{(\fB,a,\delta)} \Vert$ for almost all $a \in (\ZZ/N\ZZ)^n$.

\begin{lemma}\label{lem:upper_bound_lambdan}
	For $\varepsilon \in (0,1)$, let 
	$\ell_{\varepsilon}:= \ell_{\varepsilon}(n,N) =
	3n(\pi^2/(6\varepsilon))^{1/n} N^{1-1/n}
	$. Then
	\begin{equation*}
	\frac{1}{N^n}\cdot {\#\{a \in (\ZZ/N\ZZ)^n \ | \ \Vert u_n^{(\fB,a,\delta)} \Vert <
		c^{(m-1)/2}	n \mu^2 \ell_{\varepsilon}\}}\geq 1-\varepsilon \ .
	\end{equation*}
\end{lemma}
\begin{proof}
	Let $a \in (\ZZ/N\ZZ)^n$. As $Q(a) \subseteq \cM(a)_N$, we have 
	\begin{equation}\label{eq:pf2_boundun}
	\Vert u_n^{(\fB,a,\delta)} \Vert \leq c^{(m-1)/2} \lambda_n(Q(a)) \ .
	\end{equation}
	The lattice $Q(a)$ contains the $n$ ``short'' vectors 
	$q_1 = c_{\fB,N}^{-1}(x^{(1)}),\ldots,q_n = c_{\fB,N}^{-1}(x^{(n)})$ with 
	$\Vert x^{(j)} \Vert = \lambda_j((\ZZ a)_N)$ for $1 \leq j \leq n$.
	With $\fB = \{v_1,\ldots,v_n\}$, we can write, for every $1 \leq j \leq n$, 
	$
	q_j = \sum_{i=1}^n x^{(j)}_i v_i  
	$
	with 
	$x^{(j)} = (x^{(j)}_1,\ldots, x^{(j)}_n) \in (\ZZ a)_N$.
	Therefore, for all $1 \leq j \leq n$, 
	\begin{eqnarray}
	\Vert q_j \Vert \leq \sum_{i=1}^n |x^{(j)}_i|\Vert v_i \Vert \leq \sum_{i=1}^n  \lambda_j((\ZZ a)_N)
	\Vert v_i\Vert \leq \lambda_n((\ZZ a)_N) \sum_{i=1}^n \Vert v_i \Vert^2 \ .
	\end{eqnarray}
	This implies, since $\fB$ is $\mu$-small,
	\begin{equation}\label{eq:bound_lam_nQ}
	\lambda_n(Q(a)) \leq \max_{1\leq j \leq n}\Vert q_j \Vert \leq \lambda_n((\ZZ a)_N)  n \mu^2 \ .
	\end{equation}
	Thm.~\ref{theo:transference} applied with 
	$\Lambda = (\ZZ a)_N$ and $\Lambda^\vee = N^{-1} (\ZZ a)^{\perp_N}$ implies that
	\begin{equation*}
	\lambda_n((\ZZ a)_N)  \leq \frac{n N}{\lambda_1((\ZZ a)^{\perp_N})} \ .
	\end{equation*}
	By Lem.~\ref{lem:lower_bound_lambda1}, $\lambda_1((\ZZ a)^{\perp_N}) > 
	k_\varepsilon = \frac{1}{3}(6\varepsilon/\pi^2)^{1/n} N^{1/n}$ 
	for at least $(1-\varepsilon)N^n$ choices of $a \in (\ZZ/N\ZZ)^n$.
	Therefore,
	$
	\lambda_n((\ZZ a)_N)  <  nN/k_\varepsilon = 3n(\pi^2/(6\varepsilon))^{1/n} N^{1-1/n} 
	= \ell_\varepsilon
	$
	for at least $(1-\varepsilon)N^n$ choices of $a \in (\ZZ/N\ZZ)^n$.
	The bound for $\Vert u_n^{(\fB,a,\delta)} \Vert$ then follows by combining \eqref{eq:pf2_boundun} and \eqref{eq:bound_lam_nQ}. \qed
\end{proof}

\subsubsection*{Step 2.}
\label{sss:Step2_II}

We now compute a lower bound for the right-hand side of 
the formula in Prop.~\ref{prop:size_algo2}, for every $a \in (\ZZ/N\ZZ)^n$. We clearly have:
\begin{equation}\label{eq:pf2_rhs}
\frac{1}{\mu^n}\cdot \frac{\Vol(\calL(\fB)_N)}{\prod_{i=n+2}^m \lambda_i(\cM(a)_N)} 
\geq \frac{1}{\mu^n}\cdot \frac{N^{m-n}}{\prod_{i=n+2}^m \lambda_i(\cM(a)_N)} \ .
\end{equation}
Since $N\ZZ^m \subseteq \cM(a)_N$, we have $\lambda_i(\cM(a)_N) \leq N$ for every $1\leq i \leq m$. 
Thereby, we have
$\prod_{i=n+2}^m \lambda_i(\cM(a)_N) \leq N^{m-n-1}$, 
which in Eq.~\eqref{eq:pf2_rhs} gives, for every $a \in (\ZZ/N\ZZ)^n$:
\begin{equation}\label{eq:low_bd_rhs_II}
\frac{1}{\mu^n}\cdot \frac{\Vol(\calL(\fB)_N)}{\prod_{i=n+2}^m \lambda_i(\cM(a)_N)} 
\geq 
\frac{N}{\mu^n} \ .
\end{equation}

\subsubsection*{Step 3: Proof of Theorem \ref{theo:main_II}.}
\label{sss:proof_main_theorem_II}
Let $\fB \in \Omega$ and $\varepsilon \in (0,1)$. 
The assumption in Eq.~\eqref{eq:condition_main_II} implies
that
$
\log(N\varepsilon) > \frac{(m-1)n}{2}\log(c) + n(n+2)\log(\mu) + n\log(3n^2) +\log(\pi^2/6)
$,
which by a direct computation is equivalent to 
\begin{equation}\label{eq:main_proof_II}
c^{(m-1)/2} n \mu^2 \ell_\varepsilon < \frac{N}{\mu^n} \ ,
\end{equation}
where 
$\ell_\varepsilon = 3 n (\pi^2/(6\varepsilon))^{1/n} N^{1-1/n}$ 
is as in Lem.~\ref{lem:upper_bound_lambdan}.
By Lem.~\ref{lem:upper_bound_lambdan}, the left-hand side is an upper bound for 
$\Vert u_{n}^{(\fB,a,\delta)}\Vert$ for at least $(1-\varepsilon)N^n$ of the choices of $a$, 
where $\{u_i^{(\fB,a,\delta)}\}_i$ is a $\delta$-LLL reduced basis of $\cM(a)_N$.
By Eq.~\eqref{eq:low_bd_rhs_II} the right-hand side is a lower bound for 
$\frac{1}{\mu^n}\cdot \frac{\Vol(\calL(\fB)_N)}{\prod_{i=n+2}^m \lambda_i(\cM(a)_N)}$, for every $a \in (\ZZ/N\ZZ)^n$.
Hence, for at least $(1-\varepsilon)N^n$ of $a \in (\ZZ/N\ZZ)^n$, Eq.~\eqref{eq:condition_main_II} and Prop.~\ref{prop:size_algo2} give
$u_i^{(\fB,a,\delta)} \in \overline{\calL(\fB)}$ for all $1\leq i \leq n$.
This terminates the proof.

\subsection{Comparison}
\label{ss:compare_theory_heuristic}

We first compare Thm.~\ref{theo:main_I} and Thm.~\ref{theo:main_II}.
Table \ref{tb:asymptotic_logN_proven} summarizes the asymptotic lower bounds for 
$\log(N\varepsilon)$ as $n \to \infty$.
It appears that Algorithm II achieves slightly better bounds. 

\begin{table}[H]
	\begin{center}
		\renewcommand{\arraystretch}{1.3}
		\begin{tabular}{c||c|c}
			& \multicolumn{2}{c}{$\log(N)$} \\
			\hline 
			$m$ &~Algorithm I (Thm.~\ref{theo:main_I})~& ~Algorithm II (Thm.~\ref{theo:main_II})~\\[0.2pt]
			\hline
			$n+O(1)$ 	& ~$O(n^2\log(\mu))$~  &  ~$O(n^2\log(\mu))$~ \\
			$O(n)$		 &  ~$O(n^2\max(\log(\mu),\log(n)))$~ & ~$O(n^2\log(\mu))$~   \\
			$O(n^\ell) , \ell \in \RR_{>1}$ & ~$O(n^2\max(\log(\mu),n^{\ell-1} \log(n)))$~   & 
			~$O(n^2\max(\log(\mu),n^{\ell-1}))$~\\
		\end{tabular}
	\end{center}
	\caption{Asymptotic lower bounds for $\log(N)$ as functions of $n,\mu$}
	\label{tb:asymptotic_logN_proven}
\end{table}

We compare Thm.~\ref{theo:main_I} and Thm.~\ref{theo:main_II} 
with the heuristic 
estimates in Sec.~\ref{s:heuristic} (with $r=1$).
The terms in $\log(c)$ are to be compared with those in $\log(\iota)$.
As our proofs build upon 
non-tight upper bounds (e.g.~Minkowski bounds, or the number of integral points in spheres), 
our proven formulae are expectedly weaker.
The main difference between our heuristic and theoretical lower bounds for $\log(N)$ occurs in the term containing $\log(\mu)$.
In the case of Algorithm I, this difference comes from our upper bound 
for the last minimum of $\calL^\perp$ by Minkowski's Second Theorem in Eq.~\eqref{eq:pfI_last_min_Lperp}.
In the case of Algorithm II, this difference comes from our upper bound for 
$\prod_{i=n+2}^{m} \lambda_i(\cM(a)_N)$ in Eq.~\eqref{eq:low_bd_rhs_II}.


\section{Variations of the HLP}
\label{s:variations}

In this section we consider variations of the problem in Def.~\ref{defi:HLP}.
We first consider a variant including noise vectors.
Next, we consider a decisional version.

\subsection{HLP with noise}
\label{ss:noisy_HLP}

\begin{definition}\label{defi:noisyHLP}
	Let $\mu,\rho \in \RR_{>0}$, integers $1 \leq r \leq n \leq m$
	and $N$. 
	Let $\calL \subseteq \ZZ^m$ be a 
	$\mu$-small lattice of rank~$n$
	and $\{w_j\}_{j=1,\ldots,r}$ be linearly independent vectors in $\ZZ^m$
	such that there exist linearly independent vectors 
	$\{x_j\}_{j=1\ldots,r}$ in $\ZZ^m$ satisfying
	$w_j -x_j \in \calL \pmod{N}$
	and $\Vert x_j \Vert \leq \rho$ for all $j$.
	The {\em Noisy Hidden Lattice Problem (NHLP)} is the task to compute from the knowledge of 
	$n,N$ and the vectors $\{w_j\}_j$,
	a basis of the completion of any lattice $\Lambda$ satisfying the properties of $\calL$.
\end{definition}

We solve the NHLP by reducing it in the first place to a HLP.
Let $\cX$ be the rank-$r$ lattice generated by $\fX=\{x_j\}_j$; 
let $\fB$ be a $\mu$-small basis of $\calL$.
We assume that $\calL \cap \cX = \{0\}$, so that  
$\calL \oplus \cX$ has rank $n+r$.
By assumption, 
$\calL \oplus \cX$ has size 
$\sigma(\fB\cup\fX) \leq \sigma(\fB)+\sigma(\fX)\leq \mu + \rho$
and contains $\{w_j\}_j$ modulo $N$.
Therefore, the vectors $\{w_j\}_j$ are an instance of HLP with hidden lattice $\calL \oplus \cX$.\\

We first treat the special case when $\rho$ 
is larger than $\mu$.
The application of either Algorithm I or Algorithm II to $\{w_j\}_j$, 
reveals 
a reduced basis of $\overline{\calL \oplus \cX}$ if the parameters are suitable.
If $\rho$ is larger than $\mu$, 
one can distinguish, in a reduced basis of 
$\overline{\calL \oplus \cX}$, 
the vectors of $\overline{\calL}$ from 
those of $\overline{\cX}$. \\

In the general case, without the assumption $\rho > \mu$,
we do not expect a significant 
gap between vectors of $\overline{\calL}$ and $\overline{\cX}$ in a reduced basis of 
$\overline{\calL \oplus \cX}$,
and thereby, cannot directly identify
$\overline{\calL}$.
We overcome this problem via an embedding in larger dimension and the resolution 
of a system of linear equations.
More precisely, 
let $\calL' \in \ZZ^{m+r}$ be embedded in $\ZZ^{m+r}$ as 
$(\calL,0)$,
that is, the vectors $(v,0) \in \calL\times \{0\}^r$.
For $1\leq j\leq r$, let 
$w_j'=(w_j,e_j) \in \ZZ^{m+r}$
and $x_j' = (x_j,e_j) \in \ZZ^{m+r}$, 
where 
$e_j \in \ZZ^r$ is the $j$th standard unit vector;
let $\cM' \subseteq \ZZ^{m+r}$ be the rank-$r$ 
lattice 
generated by $\{w_j'\}_{j}$,
and $\cX' \subseteq \ZZ^{m+r}$ the 
rank-$r$ lattice generated by $\{x_j'\}_j$.
Clearly, 
$\cM' \subseteq \calL' \oplus \cX' \pmod{N}$,
and $\calL' \oplus \cX'$ is a small hidden lattice
of rank $n+r$ in dimension $m+r$.
We proceed as follows to compute
$\overline{\calL}$.
Let $\pi : \ZZ^{m+r} \to \ZZ^r$ be the projection 
onto the last $r$ coordinates.
We can distinguish between vectors in
$\overline{\calL}$ and $\overline{\cX}$ 
by noticing that
for every $v \in \calL'$, it holds $\pi(v) = 0$,
and $\cX' \cap \{v \in \ZZ^{m+r} : \pi(v)=0\} = \{0\}$.
We consequently recover (basis) vectors $v$ of $\overline{\calL}$ from 
vectors in
$\overline{\calL'}$ by solving a system of linear equations, imposing the condition
$\pi(v)=0$.

Practically,
let $B$ be a reduced basis matrix
of 
$\overline{\calL' \oplus \cX'}$,
say
$
B = [V | U] 
\in \ZZ^{(n+r)\times(m+r)} 
$,
where
$V \in \ZZ^{(n+r) \times m}$
and 
$U \in \ZZ^{(n+r) \times r}$,
and computed by 
either Algorithm I or II,
on input $\cM'$. 
By Sec.~\ref{s:heuristic}, we expect to compute such $B$ 
successfully under the 
heuristic conditions 
\eqref{eq:logN_Algo1}
and 
\eqref{eq:heuristic_logN_Algo2_without_prop},
with, essentially, replacing
$n$ by $n+r$, 
$m$ by $m+r$
and $\mu$ by $\mu+\rho$.
We next 
compute the left-kernel of $U$, that is,
$K \in \ZZ^{n \times (n+r)}$
such that $KU=0_{n,r}$.
This implies
$
KB = [KV|0_{n,r}]
$.
Heuristically, the rows in $KB$ must be in $\overline{\calL'}$,
as the last $r$ components are zero.
Then the rows of 
$KV$ form a basis for $\overline{\calL}$.
Namely, we heuristically expect to uniquely recover 
$\overline{\calL'}$,
as it is unlikely in the ``generic'' case,
that there exists a small lattice 
$\Lambda \neq \calL$ of rank $n$ in $\ZZ^m$ 
such that 
$\Lambda \oplus \cX$ contains 
$\cM$ modulo $N$.

\begin{algorithm}
	\caption{Solve the NHLP in general}
	\label{alg:nhlp_system}
	\flushleft {\bf Parameters:}
	The HLP parameters $n,m,r,\mu,\rho,N$ from Def.~\ref{defi:noisyHLP}\\
	{\bf Input:}
	A valid input for the NHLP\\
	{\bf Output:}		
	A basis of the lattice $\overline{\calL}$ (under suitable parameter choice)\\
	(1)	Run Algorithm I or Algorithm II on the lattice $\cM' \subseteq \ZZ^{m+r}$ generated by $\{w_j'\}_j$; write the output basis vectors into the rows of a matrix
	$B \in \ZZ^{(n+r)\times(m+r)}$ \\
	(2)
	Write  
	$B = [V | U]$
	with 
	$V \in \ZZ^{(n+r) \times m}$
	and 
	$U \in \ZZ^{(n+r) \times r}$.
	Compute $K \in \ZZ^{n \times (n+r)}$
	such that $KU=0_{n \times r}$\\
	(3) Return the basis given by the rows of $KV$
\end{algorithm}

\subsection{Decisional HLP}
\label{ss:decisional_HLP}

\begin{definition}\label{defi:Decision_HLP}
	Let $\mu \in \RR_{\geq 1}$, integers $1 \leq r \leq m$ 
	and $N \in \ZZ$. 
	Let $\cM \subseteq \ZZ^m$ be a lattice
	of rank $r$.
	The {\em Decisional Hidden Lattice Problem (DHLP)} is the task to decide from
	the knowledge of $\mu,N$ and a basis of $\cM$, 
	whether there exists a $\mu$-small lattice $\calL \subseteq \ZZ^m$ of rank $1 \leq n \leq m$
	such that $\cM \subseteq \calL \pmod{N}$.
\end{definition}

The rank of $\calL$ is not given as input, and our algorithm is able to detect it. 
Note that when $\calL$ exists, then there exist many small lattices of lower ranks (e.g.~the sublattices of $\calL$). Therefore, we would like $\calL$ to be of maximal rank.

\subsubsection{A geometric approach.}
\label{sss:decisional_gaps}
To solve the DHLP for $\cM$ and $N$, we consider 
the successive minima of $\cM_N$ (resp.~$\cM^{\perp_N}$) and show that the existence of a small lattice 
$\calL$ impacts the geometry of $\cM_N$ (resp.~$\cM^{\perp_N}$).
Lattices with gaps in their minima and the impact on cryptosystems are for example studied in \cite{gaps}. 

\begin{lemma}\label{lem:gaps}
	For every lattice $\Lambda \subseteq \ZZ^m$ of rank $m$ and 
	every sublattice $\Lambda' \subseteq \Lambda$ of rank $0< m'< m$, 
	one has
	$
	\frac{\prod_{k=m'+1}^m \lambda_k(\Lambda)}{\prod_{k=1}^{m'}
		 \lambda_k(\Lambda)} \geq \gamma_{m'}^{-m'}
	  \frac{\Vol(\Lambda)}{\Vol(\Lambda')^2}
	$.
\end{lemma}

\begin{proof}
	The quotient can be written as 
	$\prod_{k=1}^m \lambda_k(\Lambda)/(\prod_{k=1}^{m'}\lambda_k(\Lambda))^2$.
	
	One has 
	$\prod_{k=1}^m \lambda_k(\Lambda)\geq \Vol(\Lambda)$.
	The denominator is at most
	$(\prod_{k=1}^{m'} \lambda_k(\Lambda'))^2$
	as $\Lambda'$ is a sublattice of $\Lambda$. Finally, Eq.~\eqref{eq:minko_2} gives the result. \qed
\end{proof}

\begin{corollary}\label{cor:gaps}
	Let $\cM \subseteq \ZZ^m$ be a lattice of rank $r$ and $N>0$ an integer.
	Assume that $\Vol(\cM^{\perp_N}) = N^r$.
	If there exists a $\mu$-small lattice $\calL \subseteq \ZZ^m$ of rank $n$
	such that $\cM \subseteq \calL \pmod{N}$, then 
	\begin{equation}\label{eq:gaps_minima}
	\frac{\prod_{k=m-n+1}^m \lambda_k(\cM^{\perp_N})}{\prod_{k=1}^{m-n}
		\lambda_k(\cM^{\perp_N})} \geq 
	\gamma_{m-n}^{-(m-n)} \frac{N^r}{\mu^{2n}} \ .
	\end{equation}
\end{corollary}
\begin{proof}
	Lem.~\ref{lem:gaps} applied to the lattices $\Lambda = \cM^{\perp_N}$ 
	and $\Lambda' = \calL^{\perp}$ of rank $m-n$ gives 
	the lower bound $\gamma_{m-n}^{-(m-n)} \Vol(\cM^{\perp_N})/\Vol(\calL^\perp)^2$ on the considered ratio.
	We conclude using
	$\Vol(\cM^{\perp_N}) = N^r$,
	and $\Vol(\calL^{\perp}) \leq \Vol(\calL) \leq \mu^n$ by Lem.~\ref{lem:lattice}. \qed
\end{proof}

A similar result holds for
$\cM_N$
by either using Lem.~\ref{lem:gaps} with 
$\Lambda' = Q$
or invoking Banaszczyk’s Thm.~\ref{theo:transference}.
We observe that this ratio grows as $N$ gets larger.

\subsubsection*{Non-HLP instances.} 
We compare with lattices $\cM$ not lying in a
$\mu$-small lattice $\calL$
modulo $N$ (we call this a Non-HLP instance). 
Expectedly, this holds for random lattices
$\cM$, 
when $r$ basis 
vectors are uniformly chosen from $(\ZZ/N\ZZ)^m$
and we rely on the Gaussian Heuristic \eqref{eq:gauss_heuristic}.
If $\Vol(\cM^{\perp_N})=N^r$ (which is likely for random $\cM$), the minima are heuristically 
$\sqrt{m/(2\pi e)} N^{r/m}$.
Therefore, for any $1 \leq n \leq m-1$:
\begin{equation}\label{eq:gaps_heuristic}
\frac{\prod_{k=m-n+1}^m \lambda_k(\cM^{\perp_N})}{\prod_{k=1}^{m-n}
\lambda_k(\cM^{\perp_N})} 
\gtrapprox 
\frac{(\sqrt{m/(2\pi
 e)}N^{\frac{r}{m}})^{n}}{(\sqrt{m/(2\pi e)}N^{\frac{r}{m}})^{m-n}}
= \sqrt{\frac{m}{2\pi e}}^{2n-m}N^{\frac{r(2n-m)}{m}} \ .
\end{equation}
For $m=2n$, this approximation is $1$,
and much larger 
if $2n > m$. In particular, we observe that \eqref{eq:gaps_heuristic} 
is in general much smaller than \eqref{eq:gaps_minima}, 
as can be seen when choosing $m > 2n$ and relatively small values of $\mu$. 

\subsubsection{Heuristic Algorithm for DHLP.}
\label{sss:decisional_algo}

Since we cannot compute the successive
minima efficiently, 
the ratio in Cor.~\ref{cor:gaps}
is not practical. 
Instead, we approximate the minima 
by the norms of the vectors in an LLL-reduced basis.
Using Thm.~\ref{theo:LLL}, it is immediate to establish 
a similar lower bound for the ratio 
$$g_{m-n}(\cM^{\perp_N}):=\frac{\prod_{k=m-n+1}^m \Vert u_k \Vert}{\prod_{k=1}^{m-n}	\Vert u_k \Vert} $$
where 
$\{u_k\}_k$ is an LLL-reduced basis of $\cM^{\perp_N}$.
Such a lower bound gives a necessary condition 
for the existence of 
a $\mu$-small lattice $\calL$ such that $\cM \subseteq \calL \pmod{N}$. 
Eq.~\eqref{eq:gaps_minima} shows an explicit dependence on $n$, the rank of $\calL$.
Since $n$ is unknown, 
one first detects $m-n$ (the rank of $\calL^\perp$) by computing the successive ratios 
$\{g_{m-\ell}(\cM^{\perp_N})\}_\ell$ defined by 
$g_{m-\ell}(\cM^{\perp_N}) = \prod_{k=m-\ell+1}^m \Vert u_k \Vert/\prod_{k=1}^{m-\ell} \Vert u_k \Vert$ for 
$\ell = 1,\ldots,m-1$ and $\{u_k\}_k$ a reduced basis of $\cM^{\perp_N}$; one has 
$g_1(\cM^{\perp_N}) \geq g_2(\cM^{\perp_N}) \geq \ldots \geq g_{m-1}(\cM^{\perp_N})$.
One then identifies the smallest index $m-\ell_0$ such that $g_{m-\ell_0}(\cM^{\perp_N})$ is significantly larger 
than $g_{m-\ell}(\cM^{\perp_N})$ 
for all $\ell < \ell_0$.
In that case, we expect the existence of a hidden 
small lattice of rank $n=\ell_0$.
Again, this is easily adapted for Algorithm II when considering $\cM_N$ instead of $\cM^{\perp_N}$.
Although this approach only solves DHLP in one direction, 
we heuristically expect the converse to be true: 
if these gaps are sufficiently large, 
then there exists a small lattice $\calL$ containing $\cM$ modulo $N$.

\section{Applications and Impacts on Cryptographic Problems}
\label{s:applications}

In this section we address applications of the 
hidden lattice problem in 
cryptography
and discuss the impact of our algorithms.
In the literature, these problems are typically solved by means of Algorithm I. 
Our Algorithm II provides a competitive alternative for solving these problems.

\subsection{CRT-Approximate Common Divisor Problem}
\label{ss:crtacd}

The Approximate Common Divisor Problem based on Chinese Remaindering 
can be stated as follows (e.g.~\cite[Def.~3]{corper} or\cite[Def.~5.1]{sim_diag_incomplete}):
\begin{definition}\label{def:crtacd}
	Let $n, \eta, \rho \in \ZZ_{\geq 1}$.
	Let $p_1,\ldots,p_n$
	be distinct $\eta$-bit prime numbers
	and $N=\prod_{i=1}^{n} p_i$.
	Consider a non-empty finite set
	$\mathcal{S} \subseteq \ZZ\cap[0,N)$
	such that for every $x \in \mathcal{S}$:
	$$x \equiv x_{i} \pmod{p_i} \ , \quad 1 \leq i \leq n $$
	for 
	integers $x_{i} \in \ZZ$ satisfying
	$|x_{i}|\leq 2^\rho$.
	
	The {\em CRT-ACD} problem states as follows:
	given the set $\mathcal{S}$,
	the integers $\eta,\rho$ and $N$, factor $N$ completely
	(i.e.~find the prime numbers $p_1,\ldots,p_n$).
\end{definition}

An algorithm for this problem was described in 
\cite{corper} for $\#\cS = O(n)$,
and improved in 
\cite{sim_diag_incomplete} to $\#\cS = O(\sqrt{n})$.
These algorithms build on two steps, 
where the first step agrees and is based on solving a hidden lattice problem.

\subsubsection{HLP and algorithms for the CRT-ACD problem.}
We follow \cite{corper} to recall the first step of the algorithm.
Let $\cS=\{x_1,\ldots,x_n,y\}$ and $x = (x_1,\ldots,x_n) \in \cS^n$,
with $\#\cS = n+1$.
The vector 
$b = (x,y\cdot x) \in \ZZ^{2n}$ is public, and
by the Chinese Remainder Theorem, letting 
$x \equiv x^{(i)} \pmod{p_i}$
and $y \equiv y^{(i)} \pmod{p_i}$ 
for all $1\leq i \leq n$, 
one has 
$$b \equiv \sum_{i=1}^n c_i (x^{(i)},y^{(i)}x^{(i)}) =: \sum_{i=1}^n c_i b^{(i)} \pmod{N}\ ,$$
for some integers $c_1,\ldots,c_n$.
If $\{x^{(i)}\}_i$ are $\RR$-linearly independent, then so are 
$\{b^{(i)}\}_i$ and generate a $2n$-dimensional lattice $\calL$ of rank $n$.
Importantly, by Def.~\ref{def:crtacd}, 
$\{b^{(i)}\}_i$ are reasonably short vectors with entries bounded by $2^{2\rho}$, approximately.
The basis $\{b^{(i)}\}_i$ of $\calL$ has size
$\mu := \sigma(\{b^{(i)}\}_i) = 
(n^{-1}\sum_{i=1}^{n} \Vert b^{(i)} \Vert^2)^{1/2} \lessapprox \sqrt{2n}\cdot 2^{2\rho}$, i.e.~$\mu = O(n^{1/2}2^{2\rho})$. 
Since the basis $\{b^{(i)}\}_i$ of $\calL$ is secret, and 
$b \in \calL \pmod{N}$, we view the vector $b$ 
(or rather, the rank-one lattice $\cM = \ZZ b$) 
as an instance of a 
HLP of rank $r=1$, with hidden lattice $\calL$ of size $O(n^{1/2} 2^{2\rho})$.
Based on this observation, the algorithms in \cite{corper, sim_diag_incomplete} rely on the orthogonal lattice 
attack to compute a basis of $\overline{\calL}$.
Therefore, following Algorithm I, the first step is to run lattice reduction on $\cM^{\perp_N}$,
and construct the sublattice $\cN_{\rm I}$ of $\calL^\perp$.
Upon recovery of such a basis, the authors proceed with an 
``algebraic attack'', 
based on computing the eigenvalues of a well-chosen (public) matrix and 
then revealing the prime numbers $\{p_i\}_i$ by a gcd-computation.

\subsection{The Hidden Subset Sum Problem}
\label{ss:hssp}

The definition of the hidden subset sum problem (HSSP), 
as considered in \cite{ns99, hssp_cg20}, is as follows:

\begin{definition}\label{def:hssp}
	Let $n, m \in \ZZ_{\geq 1}$ with $n \leq m$,
	and $N \in \ZZ_{\geq 2}$.
	Let $v \in \ZZ^m$ be such that 
	$v \equiv \sum_{i=1}^n \alpha_i x_i \pmod{N}$
	with $\alpha_i \in \ZZ$ and $x_i \in \{0,1\}^m$.
	
	The problem states as follows:
	given $v$ and $N$,
	compute vectors $\{x_i\}_i \in \{0,1\}^m$ and 
	integers $\{\alpha_i\}_i$ such that 
	$v \equiv \sum_{i=1}^n \alpha_i x_i \pmod{N}$. 
	
\end{definition}

\subsubsection{HLP and algorithms for HSSP.}
In \cite{ns99}, Nguyen and Stern describe 
an algorithm  
in two steps.
The first step solves a HLP with $r=1$:
namely, 
the lattice 
$\calL:= \bigoplus_{i=1}^n \ZZ x_i$
is $\mu$-small for $\mu=\sigma(\{x_i\}_i) \leq \sqrt{m}$, 
and $v \in \calL \pmod{N}$.
Nguyen and Stern, therefore compute 
a basis of $\overline{\calL}$, 
by the orthogonal lattice algorithm.
The second step reveals 
$\{x_i\}_i$ and $\{\alpha_i\}_i$ 
from such a basis.
Recently, Coron and Gini \cite{hssp_cg20} 
argued that, due to the need of running 
BKZ with increasingly 
large block-size to compute a short basis,
the second step of the algorithm in
\cite{ns99} has exponential
(in $n$) complexity, and is therefore practical only in low dimension.
Moreover, \cite{hssp_cg20} gives 
an alternative second step, 
based on solving a system of multivariate equations, 
running in polynomial-time,
at the cost of a larger dimension
$m=O(n^2)$.
The first step, i.e.~the resolution of the associated
HLP, remains unchanged
and is solved by Algorithm I.
In dimension
$m=O(n^2)$, lattice reduction 
becomes unpractical. Therefore the authors 
employ a technique to compute $\overline{\calL}$ when $m$ is a lot larger
than $n$ (see \cite[Sec.~4.1, Sec.~5]{hssp_cg20}),
based on computing a reduced basis
for $(\ZZ v)^{\perp_N} \subseteq \ZZ^m$
by parallelizing lattice reduction over several components of $v$, of smaller dimension, say $2n$. 
The idea is to reduce the HLP $\cM \subseteq \calL \pmod{N}$ 
in dimension $m$, to multiple HLP's $\pi_j(\cM) \subseteq \pi_j(\calL) \pmod{N}$,
where $\{\pi_j\}_j$ are projections $\ZZ^m \to \ZZ^{m'}$ onto block-components, 
where $m'<m$ (e.g.~when $m=O(n^2)$, one can let $m'=O(n)$ and $O(n)$ projections $\{\pi_j\}_j$).
One then 
reconstructs a full basis of $\overline{\calL}$.
We note that it is immediate to adapt this method to our Algorithm II.

In \cite{ns99,hssp_cg20}, the density
is defined as $n/\log(N)$, as analogy to the 
classical subset sum problem \cite{cssplo} (see Rem.~\ref{rem:density}).
When $m=2n$, the density is heuristically at most $O(1/n)$ 
and proven $O(1/(n\log(n)))$ in \cite{hssp_cg20}.

Our Algorithm II can in turn be used to solve the HLP in the first step of the algorithms of \cite{ns99,hssp_cg20}.
Note that when $m=2n$ (and $\mu = O(\sqrt{n})$), the density is heuristically at most $O(1/n)$ 
and proven $O(1/(n\log(\sqrt{n})))$ according to Table \ref{tb:asymptotic_logN_proven}.
This gives a factor $2$ improvement compared to \cite{hssp_cg20}.

\subsection{More applications related to Cryptography}

\subsubsection*{CLT13 Multilinear Maps.} The work \cite{cltindslots} studies the security of CLT13 
Multilinear Maps \cite{clt13} with independent slots.
This is an example of our NHLP from Def.~\ref{defi:noisyHLP}, as we now explain; we refer to \cite{cltindslots} for details.
The attacker derives equations 
$w_k \equiv \sum_{i=1}^\theta \alpha_{ik} m_{i} + R_k \pmod{x_0}$, for $1 \leq k \leq d$,
where $\{w_k\}_j \subseteq \ZZ^\ell$ are public vectors (corresponding to zero-tested encodings), 
$\{\alpha_{ik}\}_{i,k}$ unknown integers, 
$\{m_{i}\}_i$ short secret plaintext 
vectors, and 
$\{R_k\}_k$ unknown ``noise'' vectors.
Here, $x_0$ is a public integer.
We interpret this directly as a NHLP with $r=d$: namely
$w_k-R_k \in \calL:=\bigoplus_{i=1}^\theta \ZZ m_i \pmod{x_0}$, for all $1 \leq k \leq d$,
and the basis $\{m_i\}_i$ of $\calL$ is ``small''.
As noticed in \cite{cltindslots}, 
the $\{\alpha_{ik}\}_{i,k}$ carry a special structure, making
the attack more
direct. The first step of Algorithm I for the NHLP 
(Sec.~\ref{ss:noisy_HLP}),
reveals a basis of the lattice $\Lambda$ 
of vectors orthogonal to $\{m_{i}\}_i$ modulo certain small 
primes $\{g_i\}_i$ defining the plaintext ring.
Therefore, we view $\Lambda$ as hidden lattice, 
rather than $\calL$.
Upon computing a basis of $\Lambda$, the attack 
proceeds by 
revealing the (secret) volume $\prod_{i=1}^\theta g_i$ of $\Lambda$.

\subsubsection{RSA Signatures.} In \cite{fault_attacks_emv_sign}, the authors describe 
a cryptanalysis on a signature scheme based on RSA, following an attack similar to
\cite{crypt_NS_SAC}.
In \cite{crypt_rsacrt}, a very similar attack is described against RSA-CRT signatures.
Following \cite[Sec.~3]{fault_attacks_emv_sign},
by considering $\ell$ faulty signatures together with a public modulus $N = pq$,
one derives an equation 
$
a_i + x_i + c y_i \equiv 0 \pmod{p}
$
for $1 \leq i \leq \ell$,
where $\{a_i\}_i$ are known integers, and $\{x_i\}_i, \{y_i\}_i, c$ are unknown. 
Letting $a = (a_i)_i$ gives $a \in \calL \pmod{p}$,
where $\calL$ is the rank-$2$ lattice $\ZZ x \oplus \ZZ y$
generated by $x=(x_i)_i$ 
and $y=(y_i)_i$ in $\ZZ^\ell$.
If $\{x_i\}_i$ and $\{y_i\}_i$ are sufficiently small, then $\calL$ is a suitably small lattice,
describing a HLP of rank $2$ in dimension $\ell$. 
The authors follow Algorithm I
to compute a basis $\{x',y'\}$ of $\overline{\calL}$.
Upon recovery of $x',y'$, 
the attack proceeds by simple linear algebra and a gcd computation to reveal $p$.

In this case, the hidden lattice has rank only $2$. 
Therefore, Algorithm II is much more direct in the second step.
While $\ell$ is not very large in \cite{fault_attacks_emv_sign,crypt_rsacrt}, we note that, in general, 
computing the completion of the rank-$2$ lattice $\cN_{\rm II}$,
is much faster than computing the orthogonal of the rank-$(\ell-2)$ lattice $\cN_{\rm I}$, 
as in \cite{fault_attacks_emv_sign, crypt_rsacrt}.

\section{Practical aspects of our algorithms}
\label{s:practical}
We provide practical results for the HLP obtained in SageMath \cite{sage_92}.
Our experiments are done on a standard 
laptop.
The source code is available at \url{https://pastebin.com/tNmgjkwJ} using the password \texttt{hlp\_}.
For $a\in \ZZ_{\geq 2}$, let $\fp(a)$ denote the smallest prime number larger than $2^a$.

\subsubsection*{Instance generation.}
We generate random instances of the HLP and test Algorithms I and II.
Given fixed integers $r,n,m,N$ as in Def.~\ref{defi:HLP}, we uniformly at random generate a 
basis $\fB$ for a hidden lattice $\calL$,
where the absolute values of the entries of each vector are bounded by some positive integer $\alpha$,
i.e.~every vector has infinity norm at most $\alpha$.
We let $\mu:=\sigma(\fB)$, then by construction, 
$\calL$ is $\mu$-small.
To generate a lattice $\cM \subseteq \calL \pmod{N}$ of rank $r$, we generate
$r$ uniformly random linear combinations modulo $N$ 
of the basis vectors in $\fB$.
For large $n,m,\mu$, the lattice $\calL$ is likely complete. 

\subsubsection*{Running times.}
In Table \ref{tb:hlp_running_times} we compare the running times for our algorithms.
Here $N=\fp(a)$ where $a$ is indicated in the column 
``$\log(N)$''.
For Algorithm I, ``Step 1'' runs LLL on $\cM^{\perp_N}$ and 
computes $\cN_{\rm I}$;
``Step 2'' computes $\cN_{\rm I}^\perp$ following Sec.~\ref{ss:practical_aspects}.
For Algorithm II, ``Step 1'' runs LLL on $\cM_N$ and computes 
$\cN_{\rm II}$, while ``Step 2'' computes 
$\overline{\cN_{\rm II}}$.
For the latter, we compute $\overline{\cN_{\rm II}}^{N^\infty}$; namely, in these cases we have 
$(\overline{\calL}:\cN_{\rm{II}}) = N^{n-r}$.
For this step, we compare the running time with Magma \cite{magma}, which seems to perform the finite field linear algebra much faster. 
The total running times for Algorithm II are therefore very competitive and constitute a major strength of Algorithm II against Algorithm I.
As observed in Sec.~\ref{ss:practical_aspects},
the running time for Algorithm II 
is largely reduced for larger values of $r$ (e.g.~$\geq m/2$). 
In these cases, Algorithm II outperforms Algorithm I.

\subsubsection*{Modulus size.}


In Table \ref{tb:hlp_mod_size}, we fix $m, r$ and $\mu$ and find, 
for increasing values of $n$, the smallest value for $\log(N)$ such that 
a randomly generated HLP with parameters $n,m,r,\mu,N$ is solvable by our algorithms.
We compare the practical values with the heuristic values from Sec.~\ref{s:heuristic}.
The columns ``heuristic'' stand for the lower bounds in \eqref{eq:logN_Algo1}, resp.~\eqref{eq:heuristic_logN_Algo2_without_prop}. 
Practically, we observe that the condition in Eq.~\eqref{eq:heuristic_II_direct_condition} is already satisfied for $\theta=1$, thus we may neglect the last term in Eq.~\eqref{eq:heuristic_logN_Algo2_without_prop}, which becomes negative.
We run LLL, so we set $\log(\iota)=0.03$.
We study two series (1 and 2) according to $m,r$.
Conjecturally, we see that the practical bound for $\log(N)$ is the same for both algorithms; 
this is to be expected from the duality (see Sec.~\ref{ss:dual_algo}).
An interesting question is to find a theoretical optimal bound for $\log(N)$ 
fitting best with the practical behaviour.

\subsubsection*{Output quality of basis.}
In random generations, $\calL$ is complete with high probability, and we compare
$\mu$ to
the size of the basis output by Algorithms I and II.
We observe that our algorithms compute much smaller (LLL-reduced) bases of $\calL$,
and in fact sometimes recover the basis uniquely (up to sign). 
In particular, they sometimes solve
the stronger version of Def.~\ref{defi:HLP}, that of computing 
a {\em $\mu$-small basis} instead of {\em any}.
Table \ref{tb:size_basis} is obtained for $m=2n=4r$
for increasing values of $m$;
in this case $\mu$ is approximately $N^{1/4}$, as predicted theoretically by Eq.~\eqref{eq:expected_size_mu}.

\subsubsection*{Decisional HLP.}
We test the decisional version of the HLP 
of Sec.~\ref{ss:decisional_HLP}. 
Table \ref{tb:gaps} shows different values for 
$g_{m-n}(\cM^{\perp_N})$ if $\cM$ lies in a $\mu$-small lattice 
$\calL$ modulo $N$ (HLP instance), 
and if $\cM$ is randomly sampled (random instance). 
For the latter,
we compare with the heuristic
bound \eqref{eq:gaps_heuristic} 
and report it in the coulmn ``heuristic''. 
We fix $n=25$ and consider increasing values of $r<35$.\\

\begin{table}
	\begin{center}
		\begin{tabular}{|c|c|c||c|c|}\cline{4-5}
			\multicolumn{3}{c|}{}  & \multicolumn{2}{c|}{Size of output basis} \\ 
			\hline
		 \ $m$ \ & \ $\log(N)$ \ & \ $\log(\mu)$ \ & ~Algorithm I~ &  ~Algorithm II~\\
			\hline
			$76$	& $175$ & $42.335$ 	& $42.335$ &  ~$42.335$~ \\
			 $160$  &$325$ & $77.874$ &~$79.927$~&~$79.814$~  \\
			 $320$  &~$80$~&~$13.36$~&~$18.132$~&~$18.182$~\\
			\hline
		\end{tabular}
	\end{center}
	\caption{Sizes of output bases for Algorithms I and II}
	\label{tb:size_basis}
\end{table}

\begin{table}
	\begin{center}
		\begin{tabular}{|c|c|c||c|c|c|c|}\cline{4-7}
			\multicolumn{3}{c|}{}  & \multicolumn{2}{c|}{~HLP instance~} & \multicolumn{2}{c|}{~random instance~} \\ 
			\hline
			\ $r$ \ & \ $m$ \ & \ $\log(N)$ \ & \ $\mu$ \ & ~$g_{m-n}( \cM^{\perp_N})$~ &  ~$g_{m-n}( \cM^{\perp_N})$~&~heuristic~\\
			\hline
			$1$ & $45$ & $350$ & $581.73$ &  $3.52 \cdot 10^{21}$  & $2.74 \cdot 10^{12}$&$5.73\cdot 10^{12}$\\
			$5$	& $70$ & $100$ 	& $580.11$ &  $7.65 \cdot 10^{49}$ &
			$3.39 \cdot 10^{-56}$& $7.4\cdot 10^{-50}$\\
			$10$ &  $50$  &$100$  & $2037.15$ &  $2.95 \cdot 10^{188}$  &
			$1.27$& $1$\\
			$20$ & $80$  &$85$ & $2949.66$ & $6.06\cdot 10^{305}$   &  $3.16\cdot10^{-191}$& $4.09\cdot 10^{-180}$  \\
			\hline
		\end{tabular}
	\end{center}
	\caption{Gaps in LLL-reduced bases of
		$\cM^{\perp_N}$ }
	\label{tb:gaps}
\end{table}

\begin{table}
	\begin{center}
		\begin{tabular}{|c|c|c|c||c|c|c|c|c|}\cline{5-9}
			\multicolumn{4}{c|}{}  & \multicolumn{5}{c|}{~Running Time~}\\
			\cline{5-9}
			\multicolumn{4}{c|}{}  & \multicolumn{2}{c|}{~Algorithm I~} & \multicolumn{3}{c|}{~Algorithm II~} \\ 
			\hline
			\ $r$ \ & \ $n$ \ & \ $m$ \ &$\log(N)$&Step 1&Step 2&Step 1&Step 2 (Sage)&Step 2 (Magma) \\
			\hline
			$60$ & $150$ & $200$ &  $140$  & $7$ min $13$ s	& $1$ min $20$ s & $10$ min $2$ s& $3$ min $4$ s & $0.37$ s\\
			$110$ &$150$ & $200$ &  $90$ & $6$ min $20$ s & $1$ min $29$ s &  $4$ min  &  $1$ min $33$ s & $0.24$ s \\
			$175$&$180$ & $200$ &  $140$ & $6$ min $56$ s& $1$ min $24$ s & $1$ min $39$ s & $20$ s & $0.19$ s \\
			\hline
			$80$&$100$ & $300$ &  $75$ & $3$ min $51$ s & $30$ min $17$ s & $22$ min $51$ s & $30$ s& $0.12$ s\\
			$150$&$200$ & $300$ &  $75$ & $145$ min $29$ s & $22$ min $23$ s & $116$ min $14$ s & $6$ min $19$ s & $0.56$ s\\
			\hline
			$75$&$150$ & $400$ &  $80$ & $75$ min $16$ s & $326$ min $44$ s & $414$ min $51$ s & $5$ min $13$ s & $0.61$ s\\
			$235$&$275$ & $400$ &  $80$ & $527$ min $43$ s & $117$ min $2$ s & $304$ min $10$ s & $15$ min $39$ s & $0.95$ s\\
			\hline
		\end{tabular}
	\end{center}
	\caption{Running times for Algorithms I and II; the entries of a small basis of $\calL$ lie in $(-2^{10},2^{10})\cap\ZZ$, which gives $\log(\mu)\approx 13$ in all instances}
	\label{tb:hlp_running_times}
\end{table}

\begin{table}
	\begin{center}
		\begin{tabular}{|c|c||c|c|c|c|}\cline{3-6}
			\multicolumn{2}{c|}{} & 
			\multicolumn{4}{c|}{~$\log(N)$~} \\
			\cline{3-6}
			\multicolumn{2}{c|}{} & \multicolumn{2}{c|}{~Algorithm I~} & \multicolumn{2}{c|}{~Algorithm II~} \\ 
			\cline{2-6}
			\multicolumn{1}{c|}{} &~$n$~&~heuristic~&~practice~ &~heuristic~&~practice~ \\
			\hline
			&~$10$~&~$52$~&~$41$~&~$46$~&~$41$~   \\
			&~$20$~&~$113$~&~$92$~&~$103$~&~$92$~\\
			Series $1$	  	   	  &~$40$~&~$282$~&~$241$~&~$274$~&~$241$~ \\
			($m=100, r=5$)    &~$80$~&~$1486$~&~$1384$~&~$1643$~&~$1384$~ \\
			&~$90$~&~$3240$~&~$3075$~&~$3695$~&~$3075$~\\
			\hline
			\hline
			&~$50$~&~$57$~&~$48$~&~$54$~&~$48$~ \\
			Series $2$			  &~$100$~&~$139$~&~$121$~&~$143$~&~$121$~ \\
			($m=250, r=30$)	 &~$160$~&~$327$~&~$296$~&~$381$~&~$296$~ \\
			&~$200$~&~$676$~&~$629$~&~$857$~&~$629$~\\
			\hline
		\end{tabular}
	\end{center}
	\caption{Minimal values for $\log(N)$ as a function of the other parameters; the entries of a small basis of $\calL$ lie in $(-2^{15},2^{15}) \cap \ZZ$, which gives $\log(\mu) \approx 18$ in all instances}
	\label{tb:hlp_mod_size}
\end{table}

\newpage
\bibliography{hidden_lattice}
\bibliographystyle{alpha}

\end{document}